\documentclass[11pt, a4paper]{amsart}

\usepackage[T1]{fontenc}
\usepackage{microtype}

\usepackage{amsmath}								
\usepackage{amssymb}
\usepackage{amsthm}
\usepackage{amscd}
\usepackage{amsfonts}
\usepackage{mathptmx}
\usepackage{stmaryrd}

\usepackage{euler}
\usepackage{extarrows}
\DeclareMathAlphabet{\mathcal}{OMS}{cmsy}{m}{n}

\usepackage[colorlinks, linktocpage, citecolor = blue, linkcolor = blue]{hyperref}
\usepackage{color}

\usepackage{tikz}									
\usetikzlibrary{matrix}
\usetikzlibrary{patterns}
\usetikzlibrary{matrix}
\usetikzlibrary{positioning}
\usetikzlibrary{decorations.pathmorphing}
\usetikzlibrary{cd}

\usepackage{fullpage}
\usepackage{enumerate}

\newtheorem*{theorem*}{Theorem}

\newtheorem{maintheorem}{Theorem}[section]
\newtheorem{maincorollary}[maintheorem]{Corollary}
									
\newtheorem{theorem}{Theorem}[section]
\newtheorem{lemma}[theorem]{Lemma}
\newtheorem{proposition}[theorem]{Proposition}
\newtheorem{corollary}[theorem]{Corollary} 

\theoremstyle{definition}
\newtheorem{definition}[theorem]{Definition}

\newtheorem{remark}[theorem]{Remark}

\newtheoremstyle{myitemstyle}						
	{}			
	{}			
	{}			
	{}			
	{}			
	{.}			
	{ }			
	{}			
\theoremstyle{myitemstyle}
\newtheorem{myitemthm}{}
\newenvironment{myitem}[1]{\begin{myitemthm} \textbf{#1}}{\end{myitemthm}}

\newcommand{\twoheadlongrightarrow}{\relbar\joinrel\twoheadrightarrow}

\newcommand{\R}{\mathbb{R}}
\newcommand{\Rbar}{\overline{\mathbb{R}}}
\newcommand{\Z}{\mathbb{Z}}

\newcommand{\C}{\mathbb{C}}

\newcommand{\N}{\mathbb{N}}
\newcommand{\PP}{\mathbb{P}}

\newcommand{\G}{\mathbb{G}}

\newcommand{\Sigmabar}{\overline{\Sigma}}
\newcommand{\sigmabar}{\overline{\sigma}}
\newcommand{\Mbar}{\overline{M}}

\newcommand{\calMbar}{\overline{\mathcal{M}}}

\newcommand{\calA}{\mathcal{A}}

\newcommand{\calC}{\mathcal{C}}
\newcommand{\calD}{\mathcal{D}}

\newcommand{\calM}{\mathcal{M}}
\newcommand{\calO}{\mathcal{O}}

\newcommand{\calS}{\mathcal{S}}
\newcommand{\calT}{\mathcal{T}}
\newcommand{\calSbar}{\overline{\mathcal{S}}}
\newcommand{\Tbar}{\overline{T}}
\newcommand{\Sbar}{\overline{S}}
\newcommand{\calUbar}{\overline{\mathcal{U}}}

\newcommand{\calX}{\mathcal{X}}

\newcommand{\calTbar}{\overline{\mathcal{T}}}

\newcommand{\scrX}{\mathscr{X}}
\newcommand{\scrD}{\mathscr{D}}
\newcommand{\scrU}{\mathscr{U}}
\newcommand{\scrV}{\mathscr{V}}

\DeclareMathOperator{\Spec}{Spec}
\DeclareMathOperator{\Hom}{Hom}

\DeclareMathOperator{\Aut}{Aut}
\DeclareMathOperator{\Inn}{Inn}

\DeclareMathOperator{\val}{val}

\DeclareMathOperator{\LOG}{LOG}
\DeclareMathOperator{\FAN}{FAN}

\DeclareMathOperator{\HOM}{HOM}

\DeclareMathOperator{\Out}{Out}

\DeclareMathOperator{\PGL}{PGL}
\DeclareMathOperator{\trop}{trop}

\DeclareMathOperator{\CV}{CV}

\DeclareMathOperator{\CVbar}{\overline{CV}}

\DeclareMathOperator{\MCG}{MCG}

\title{A non-Archimedean analogue of Teichm\"uller space and its tropicalization}

\date{}

\author{Martin Ulirsch}
\address{Institut f\"ur Mathematik, Goethe--Universit\"at Frankfurt,
Robert-Mayer-Str. 6--8,
60325 Frankfurt am Main, Germany}
\email{ulirsch@math.uni-frankfurt.de}

\begin{document}

\maketitle

\begin{abstract}
In this article we use techniques from tropical and logarithmic geometry to construct a non-Archimedean analogue of \emph{Teichm\"uller space} $\overline{\calT}_g$ whose points are pairs consisting of a stable projective curve over a non-Archimedean field and a Teichm\"uller marking of the topological fundamental group of its Berkovich analytification. This construction is closely related to and inspired by the classical construction of a non-Archimedean Schottky space for Mumford curves by Gerritzen and Herrlich. We argue that the skeleton of non-Archimedean Teichm\"uller space is precisely the tropical Teichm\"uller space introduced by Chan-Melo-Viviani as a simplicial completion of Culler-Vogtmann Outer space. As a consequence, Outer space turns out to be a strong deformation retract of the locus of smooth Mumford curves in $\overline{\mathcal{T}}_g$.
\end{abstract}

\setcounter{tocdepth}{1}
\tableofcontents


\section*{Introduction}
Let $g\geq 2$ and suppose for now that we are working over $\C$. Teichm\"uller space $\calT_{g,\C}$ is the universal cover of the moduli space $\calM_g$ of smooth curves of genus $g$. It is a complex analytic space that functions as a fine moduli space of smooth curves $X$ (of genus $g$) together with a \emph{Teichm\"uller marking}, that is a an equivalence $\phi\colon \pi_1(X)\xrightarrow{\sim} \Pi_g$\footnote{Here $\pi_1(X)$ denotes the fundamental groupoid of $X$; since $X$ is connected, an equivalence $\phi\colon \pi_1(X)\xrightarrow{\sim}\Pi_g$ amounts to choosing an outer isomorphism class $\pi_1(X,x)\xrightarrow{\sim}\Pi_g$ for one (and hence all) base points $x\in X$.}, where 
\begin{equation*}
\Pi_g=\big\langle\alpha_1,\beta_1,\ldots, \alpha_g,\beta_g\big\vert [\alpha_1,\beta_1]\cdots [\alpha_g,\beta_g]=1\big \rangle
\end{equation*}
denotes the fundamental group of a fixed Riemann surface of genus $g$. 

The \emph{moduli space $M_g^{trop}$ of stable tropical curves} of genus $g$ is a combinatorial analogue of $\calM_g$ that captures the combinatorics of the dual complex of the Deligne-Mumford compactification $\calMbar_g$ of $\calM_g$. A candidate for a tropical analogue of Teichm\"uller space is \emph{Outer space} $\CV_g$ in the sense of Culler-Vogtmann (see \cite{CullerVogtmann}), which arrived in the world of mathematics well before the recent spark in interest in tropical geometry (see \cite[Section 5]{Caporaso_comparingmoduli}). Outer space is a moduli space of metric graphs $\Gamma$ together with an equivalence $\phi\colon \pi_1(\Gamma)\xrightarrow{\sim}F_g$, where $F_g$ denotes the free group on $g$ generators. In \cite{CMV} Chan, Melo, and Viviani, building on \cite{CullerVogtmann}, construct a tropical analogue $T_g^{trop}$ of Teichm\"uller space as a natural simplicial completion of Outer space $\CV_g$ by allowing contractions of loops to vertices with integer weights. 

Denote by $\calM_g^{an}$ the non-Archimedean analytic space associated to $\calM_g$ (in the sense of Berkovich), where the base field is carrying the trivial absolute value. In \cite{ACP} (based on earlier work in \cite{Viviani_tropvscomp}) Abramovich, Caporaso, and Payne describe the connection between the algebraic and the tropical moduli space: they show that the moduli space of tropical curves is the target of a natural tropicalization map $\trop_g\colon\calM_g^{an}\rightarrow M_g^{trop}$ that sends a smooth projective curve $X$ over a non-Archimedean field to its dual tropical curve $\Gamma_X$ (i.e. the minimal skeleton of $X^{an}$ decorated by certain vertex weights keeping track of the genus of components in the reduction). Moreover, it is shown that $M_g^{trop}$ may be naturally identified with the non-Archimedean skeleton of $\calM_g$ (defined with respect to the Deligne-Mumford compactification), so that $\trop_g$ is, in particular, a strong deformation retraction. Motivated by this beautiful story for $\calM_{g}$, one might be tempted to ask:

\begin{quotation}
\emph{Which space tropicalizes to tropical Teichm\"uller space?}
\end{quotation}

\subsection*{A non-Archimedean Teichm\"uller space} From now on let $K$ be an algebraically closed non-Archimedean field (not necessarily carrying the trivial absolute value). In this article, inspired by the work of Gerritzen and Herrlich on a non-Archimedean analogue of \emph{Schottky space} \cite{Gerritzen_Siegel, Gerritzen_TeichmuellerSiegel, Herrlich_Teichmueller, Herrlich_extendedTeichmueller}, we use methods from tropical and logarithmic geometry to construct a  \emph{non-Archimedean Teichm\"uller space} $\calTbar_g$. Roughly speaking, a point in $\calTbar_g$ is a stable projective curve $X$ over a non-Archimedean extension $L$ of $K$ together with an equivalence  
\begin{equation*}
\phi\colon \pi_1\big(X^{an}\big)\xlongrightarrow{\sim} F_{b_1}\ ,
\end{equation*} 
where $b_1=b_1(X^{an})$ denotes the first Betti number of $X^{an}$, thought of as a topological space.

\begin{maintheorem}\label{thm_calTbar_g}
The non-Archimedean Teichm\"uller space $\calTbar_{g}$ is an analytic Deligne-Mumford stack that is smooth, separated and without boundary. 
\end{maintheorem}

We write $\calT_g$ for the locus of smooth curves in $\calTbar_g$ and $\calTbar_g^{\textrm{Mum}}$ as well as $\calT_g^{\textrm{Mum}}$ for the locus of stable or respectively smooth Mumford curves in $\calTbar_g$.

\subsection*{Tropicalization}
Denote by $\Tbar_g^{trop}$ the canonical compactification of $T_g^{trop}$ that parametrizes tropical curves $\Gamma$ where we allow the edge lengths to be $\infty$. There is a natural tropicalization map 
\begin{equation*}\trop_g\colon \calTbar_g\longrightarrow \Tbar_g^{trop}
\end{equation*} that sends a pair $(X,\phi)$ as above to the dual tropical curve $\Gamma_X$ of $X$ together with the natural induced marking $\phi\colon \pi_1(\Gamma_X)\rightarrow F_{b_1}$ (using that $\Gamma_X$ is a strong deformation retract of $X^{an}$). We show the following analogue of the main result of \cite{ACP}.

\begin{maintheorem}\label{thm_skel=trop}
The natural tropicalization map $\trop_{g}\colon \calTbar_g\longrightarrow \Tbar_g^{trop}$ has a section that makes $\Tbar_g^{trop}$ into a strong deformation retract of $\calTbar_g$. 
\end{maintheorem}

In particular, the restriction of $\trop_g$ to the locus $\calT_g$ of smooth curves in $\calTbar_g$ induces a strong deformation retraction onto the (non-extended) tropical Teichm\"uller space $T_g$. 

At this occasion, we point out that we prove Theorem \ref{thm_skel=trop} over any algebraically closed non-Archimedean field, contrary to the main result of \cite{ACP} which is only proved over an algebraically closed field with the \emph{trivial} absolute value. For this purpose we generalize in Section \ref{section_stackyskeleton} below the construction of a Berkovich skeleton relative to a simple normal crossing divisor carried out in \cite{GRW} to normal crossing divisors on smooth Deligne-Mumford stacks with good reduction. We also generalize \cite[Theorem 1.2.1]{ACP} to any algebraically closed non-Archimedean base fields in Theorem \ref{thm_skel=tropdetail}.

Restricting the contraction in Theorem \ref{thm_skel=trop} to the locus of smooth Mumford curves we find the following.

\begin{maincorollary}\label{cor_skel=trop}
The restriction of the tropicalization map to $\calT_g^{\textrm{Mum}}$ makes the Culler-Vogtmann Outer space $\CV_g$ into a strong deformation retract of $\calT_g^{\textrm{Mum}}$. 
\end{maincorollary}  

Our construction of $\calTbar_g$ uses the tropical construction of $\calT_g^{trop}$ and, in order to lift this to the algebraic world, it uses methods from \emph{logarithmic geometry} in the sense of Fontaine-Kato-Illusie (see \cite{Kato_logstr}) and, in particular, the theory of \emph{Artin fans}, as developed in \cite{AW, ACMW, CCUW, Ulirsch_nonArchArtin}. In fact, we define a logarithmic algebraic stack $\calT_g^{log}$ as the fibered product
\begin{equation*}
\calT_g^{log}=\calM_g^{log}\times_{\calM_g^{trop}}\calT_g^{trop}
\end{equation*}
along the natural tropicalization morphism $\trop_g\colon\calM_g^{log}\rightarrow\calM_g^{trop}$ introduced in \cite{CCUW}. This way we obtain a smooth, universally closed, non-separated Deligne-Mumford stack $\underline{\calT}^{log}_g$ locally of finite type over $\Z$ in which the complement of the locus of smooth curves in $\underline{\calT}^{log}_g$ is a divisor with simple normal crossings (over $\Z$). We then define the non-Archimedean Teichm\"uller space $\calTbar_g$ as the analytic generic of fiber of the base change of $\underline{\calT}^{log}_g$ to the valuation ring $R$ of $K$.

\subsection*{Non-Archimedean Schottky space and its tropicalization}
In \cite{Gerritzen_Siegel, Gerritzen_TeichmuellerSiegel, Herrlich_Teichmueller, Herrlich_extendedTeichmueller} Gerritzen and Herrlich construct a closely related non-Archimedean analogue $S_g$ of \emph{Schottky space} over the locus of Mumford curves in $\calM_g^{an}$ (see \cite{Koebe_UniformisierungIV, Hejhal_Schottky&Teichmueller} for the original complex construction). They crucially use Mumford's non-Archimedean analogue of Schottky uniformization for maximally degenerate curves (as introduced in \cite{Mumford_uniformization}; also see \cite{FresnelvanderPut,Luetkebohmert_book} for further details). 

A point in $S_g$ is a is $\PGL_2(L)$-conjugacy class of a faithful representation $F_g\rightarrow \PGL_2(L)$ with discontinuous image for a non-Archimedean extension. Denote by $\Omega\subseteq \PP^{1,an}_L$ the open set of ordinary points of the induced operation of $F_g$ on $\PP_L^{1,an}$. Then $\Omega/F_g$ is a Mumford curve $X^{an}$ and the induced equivalence $\pi_1(X)\xrightarrow{\sim}F_g$ is a marking in the above sense. So we have a natural morphism $S_g\rightarrow \calTbar_g$ whose image is the locus $\calT_g^{\textrm{Mum}}$ of smooth Mumford curves. 

Denote by $\Sbar_g$ the natural partial compactification of $S_g$ that extends their construction by faithful and discontinuous operations on trees of projective lines (constructed in \cite{Herrlich_extendedTeichmueller} as a rigid-analytic space). Then the above morphism $S_g\rightarrow \calTbar_g$ naturally extends to a morphism $\Sbar_g\rightarrow\calTbar_g$ whose image is the locus $\calTbar_g^{\textrm{Mum}}$ of stable Mumford curves in $\calTbar_g$.

Herrlich \cite{Herrlich_personalcommunication} was already aware that there is a natural "tropicalization map" from $S_g$ to Culler-Vogtmann Outer space $\CV_g$. In this article we expand on this realization and recover his map as the composition $S_g\rightarrow\calT_g^{\textrm{Mum}}\rightarrow \CV_g$. We refer the reader to the upcoming \cite{PoineauTurchetti_SchottkyoverZ} for a detailed examination of the relationship between $S_g$ and $\calT_g^{\textrm{Mum}}$.

\subsection*{The quotient by $\Out(F_g)$} 

Consider the natural forgetful map $\calTbar_g\rightarrow \calMbar_g$ that forgets the marking. There is a natural operation of $\Out(F_g)$ on $\calTbar_g$ with respect to which the forgetful map $\calTbar_g\rightarrow \calMbar_g$ is invariant. The quotient stack $\big[\calMbar_g\big/\Out(F_g)\big]$ is not isomorphic to $\calMbar_g^{an}$, since non-Archimedean curves with non-maximal reduction will always have stabilizer groups that are not present in $\calMbar_g^{an}$. We, however, have the following weaker Theorem \ref{thm_quotient}.

\begin{maintheorem}\label{thm_quotient}
The relative coarse quotient $\calTbar_g\big/_{\calMbar^{an}_g}\Out(F_g)$ over $\calMbar_g^{an}$ is equal to $\calMbar_g^{an}$. 
\end{maintheorem}

Here the \emph{relative coarse quotient } $\calTbar_g\big/_{\calMbar^{an}_g}\Out(F_g)$ is the \emph{relative coarse moduli space} of $\big[\calTbar_{g}\big/\Out(F_g)\big]$ over $\calMbar_g^{an}$, i.e. the morphism $\big[\calTbar_{g}\big/\Out(F_g)\big]\rightarrow \calMbar_g^{an}$ is initial among all factorizations 
\begin{equation*}
\big[\calTbar_{g}\big/\Out(F_g)\big]\longrightarrow \calX\longrightarrow \calMbar_g^{an}
\end{equation*}
such that $\calX\rightarrow \calMbar_g^{an}$ is representable (see \cite[Theorem 3.1]{AOV} for the concept of relative coarse moduli spaces in the algebraic category). In our case, this means that $\calTbar_g\big/_{\calMbar^{an}_g}\Out(F_g)$ up to natural equivalence is the only analytic stack that gives rise to such a factorization.

For a stable vertex-weighted graph $G$ denote by $\calTbar_G$ and $\calUbar_G$ the affinoid domains of $\calTbar_g$ and $\calMbar_g^{an}$ respectively that parametrize stable curves of genus $g$ for whom the underlying graph of the dual tropical curve is equal to $G$. Then Theorem \ref{thm_quotient} means that we have an equivalence
\begin{equation*}
\calUbar_G\simeq\big[\calTbar_G\big/\Out(F_{b_1(G)})\big] 
\end{equation*}
for every stable vertex-weighted graph $G$ with Betti-number $b_1(G)=h^1(G)$. 

As a consequence of Theorem \ref{thm_quotient}, we finally obtain:

\begin{maincorollary}
The coarse moduli space of the quotient $\big[\calTbar_{g}\big/\Out(F_g)\big]$ is naturally isomorphic to the coarse moduli space $\Mbar_g^{an}$ of $\calMbar_g^{an}$.
\end{maincorollary}

Here the compatibility of forming coarse moduli space with taking analytification follows from \cite[Theorem 1.2.1 and 1.2.2]{ConradTemkin_algspaces} on the analytification of algebraic spaces and \'etale equivalence relations. 


\subsection*{Complements and remarks}
\begin{myitem}{}
Let $g,n\geq 0$ such that $2g-2+n>0$. Our construction admits an immediate generalization to the case curves with marked points (as introduced in \cite{Knudsen_projectivityII}). In fact, one may construct a non-Archimedean Teichm\"uller space $\calTbar_{g,n}$ parametrizing smooth projective curves over a non-Archimedean extension $L$ together with $n$ distinct marked points $p_1,\ldots, p_n\in X(L)$ and a fixed equivalence $\pi_1(X^{an})\simeq F_{b_1(X)}$ as the fiber product 
\begin{equation*}
\calTbar_{g,n}=\calTbar_g\times_{\calMbar_g^{an}}\calMbar_{g,n}^{an}\ .
\end{equation*} 
Analogues of the above results for $\calTbar_{g,n}$ immediately follow from this description and the natural forgetful map $\calTbar_{g,{n+1}}\rightarrow \calTbar_{g,n}$ functions as a universal curve.
\end{myitem}

\begin{myitem}{} In \cite{Mochizuki_padicTeichmueller} Mochizuki develops another approach to the non-Archimedean uniformization of $\calM_g$ that goes by the name \emph{$p$-adic Teichm\"uller theory} (also see \cite{Mochizuki_intropadicTeichmueller}). As explained \cite[Section 1.1]{Mochizuki_intropadicTeichmueller} this is based on a $p$-adic analogue of Fuchsian uniformization via so-called \emph{indigenous bundles}, while our approach is essentially based on Mumford's non-Archimedean analogue of Schottky uniformization (as introduced in \cite{Mumford_uniformization}). Mochizuki, in particular, argues that Mumford's uniformization "does not strongly depend on the prime $p$", since e.g. Frobenius automorphisms play no significant role. The fact that our space $\underline{\calT}^{log}_g$ is actually defined over $\Z$ is another incarnation of this heuristic. 
\end{myitem}

\begin{myitem}{}
In this article we only use the topological fundamental group of a Berkovich analytic curve $X^{an}$. So, for example, for curves with good reduction our construction introduces nothing new. The author believes there should be a "better" analogue of Teichm\"uller space over the $p$-adic numbers that uses a different analogue of the fundamental group of $X^{an}$, e.g. the \emph{tempered fundamental group} of Andr\'e \cite{Andre_periods}. In \cite[Chapter III]{Scholze_torsion} Scholze constructs a $p$-adic version of \emph{Siegel upper half plane} in the framework of perfectoid spaces. A promising direction of future research could be to use his methods to try to construct a $p$-adic analogue of Teichm\"uller space. We refer the reader to \cite{Reinecke_infinitelevel} for further indications towards an abelian version of such a construction. 
\end{myitem}

\begin{myitem}{} In an upcoming project \cite{PoineauTurchetti_SchottkyoverZ} Poineau and Turchetti generalize Gerritzen-Herrlich Schottky space $S_g$ to a hybrid (Archimedean and non-Archimedean) analytic space over $\Spec \Z$. Its fiber over a non-Archimedean place of $\Z$ is exactly the non-Archimedean Gerritzen-Herrlich Schottky space $S_g$ and over the Archimedean place it is the complex-analytic Schottky space $S_{g,\C}$ (as in \cite{Koebe_UniformisierungIV, Hejhal_Schottky&Teichmueller}). An interesting trajectory for future research would be to develop a theory of hybrid analytic stacks in order to study the hybrid analytification of $\underline{\calT}^{log}_g$, to compare it with the Poineau-Turchetti construction in \cite{PoineauTurchetti_SchottkyoverZ}, and to study its tropicalization from both an Archimedean and a non-Archimedean perspective. 
\end{myitem}

\begin{myitem}{}
In their project \cite{PoineauTurchetti_SchottkyoverZ}, Poineau and Turchetti, in particular, construct a uniformization of the universal Mumford curve. In our framework, we can recover this as follows: Denote by $\Omega_g^{trop}$ the tropical moduli stack of tuples $\big((\widetilde{\Gamma}, l), f\colon\widetilde{\Gamma}\rightarrow\Gamma\big)$ consisting of a stable tropical curve $(\widetilde{\Gamma},l)$ with one marked leg together with a length-preserving morphism $f\colon \widetilde{\Gamma}^{st}\rightarrow \Gamma$ from the stabilization $\widetilde{\Gamma}^{st}$ of $\widetilde{\Gamma}$ (without the marked leg) to a stable tropical curve $\Gamma$ of genus $g$ that is a universal cover of the underlying weighted graphs. There is a natural morphism $\Omega_g^{trop}\rightarrow \calM_{g,1}^{trop}$ that is given by sending $\big((\widetilde{\Gamma}, l), f\big)$ to $\big(\Gamma,f(l)\big)$ and we may again build its logarithmic analogue as the fibered product
\begin{equation*}
\Omega_g^{log}=\calM^{log}_{g,1}\times_{\calM_{g,1}^{trop}}\Omega_g^{trop} \ . 
\end{equation*}
If we apply the Raynaud generic fiber to $\underline{\Omega}_g^{log}$, we find a non-Archimedean analytic stack $\Omega_g$ together with morphisms
\begin{equation}\label{eq_uniformization}\begin{tikzcd}
\overline{\Omega}_{g}\arrow[rr]\arrow[rd]&&\calMbar_{g,1}^{an}\arrow[ld]\\
& \calMbar_g^{an}
\end{tikzcd}\end{equation}
that functions as a non-Archimedean uniformization of the whole universal curve $\calMbar_{g,1}^{an}\rightarrow\calMbar_g^{an}$. So there is a natural operation of $F_g$ on $\overline{\Omega}_g$ such that the relative coarse quotient $\overline{\Omega}_g\big/_{\calMbar_{g,1}^{an}}F_g$ is isomorphic to $\calMbar_{g,1}^{an}$ and the restriction of \eqref{eq_uniformization} to a smooth Mumford curve $X$ in $\calM_g^{an}$ is exactly the presentation $X^{an}\simeq \Omega/ F_g$ of $X^{an}$ as a quotient of an open subset of $\PP^{1,an}$ by a Schottky group (as in \cite{Mumford_uniformization}).

In \cite{Ichikawa_Teichmuellermodularforms} Ichikawa constructs the universal deformation of a Mumford curve via Schottky groups in order to study the extension problem for Teichm\"uller modular forms. His construction of a universal deformation may be recovered by considering the formal fiber of $\overline{\Omega}_g$ over a Mumford curve in $\calMbar_g$. 
\end{myitem}

\begin{myitem}{}\label{item_extendedTeichmueller} Of course it is natural to wonder whether there is also a connection between the classical complex analytic Teichm\"uller space $\calT_{g,\C}$ and its tropical analogue $T_g^{trop}$. In \cite{GerritzenHerrlich_extendedSchottky}, Gerritzen and Herrlich construct a smooth compactification $\calSbar_{g,\C}$ of Archimedean Schottky space, whose boundary combinatorics around the strata of maximally degenerate Riemann surfaces captures exactly the combinatorics of Culler-Vogtmann Outer space $\CV_g$. We expect that a careful recasting of the extended Schottky spaces $\calSbar_{g,\C}$ in terms of logarithmic geometry would allow us to construct a tropicalization morphism $\calS_{g,\C}^{log}\rightarrow \calT_g^{trop}$ in the category of logarithmic $\C$-analytic stacks. 

In \cite{Herrlich_extendedTeichmuellerC} Herrlich uses this compactification (and Abikoff's bordification of Teichm\"uller space \cite{Abikoff_degRS}) to construct a partial compactification of Teichm\"uller space $\calTbar_{g,\C}$ as a complex ringed space that admits an operation of the mapping class group so that the coarse moduli space of $\big[\calTbar_{g,\C}/\MCG_g\big]$ is isomorphic to the coarse moduli space $\Mbar_{g,\mathbb{C}}$ over $\C$. For every symplectic homomorphism $\psi\colon \Gamma_g\rightarrow F_g$ there is an open subset $\calTbar_g(\psi)\subseteq \calTbar_{g,\C}$ (containing $\calT_{g,\C}$ as an open and dense subset) that admits a morphism $\calTbar_g(\psi)\rightarrow\calSbar_{g,\C}$ (given by composing the Teichm\"uller marking with $\psi$). Composing such a morphism with the (conjectural) tropicalization morphism from above, would provide us with a procedure to tropicalize Teichm\"uller space $\calT_{g,\C}\subseteq \calT_g(\psi)$. 


\begin{myitem}{}
Our article is by no means the first instance where moduli functors combine both algebraic/analytic and tropical data. In particular, we refer the reader to the following works:
\begin{itemize}
\item to \cite{Yue_logCYI, Yue_logCYII} for a non-Archimedean counting of holomorphic cylinders on Calabi-Yau surfaces, to \cite{RSPWI, RSPWII} for a logarithmic/tropical reinterpretation of the Vakil-Zinger blow of moduli spaces of elliptic stable maps on toric varieties, and to \cite{Ranganathan_logGWexpansions} and \cite{Ranganathan_productformula} for an approach to a degeneration formula \cite{ACGS} and a product formula \cite{Herr_productformula} in logarithmic Gromov-Witten theory;
\item to \cite{MarcusWise, MolchoWise_logPic}, to \cite{KassPagani, Holmes, HolmesKassPagani}, as well as to \cite{MMUVW} for an approach towards constructing a compactification of the universal Jacobian and a resolution of the universal Abel-Jacobi map; 
\item to \cite{BCGGM3} for a construction of a compactification of a strata of abelian differentials using combinatorial data which may be translated into tropical language expanding on \cite{MUW}; and
\item to \cite{KKN_logAVI} for a modular interpretation of toroidal compactifications of the moduli space $\calA_g$ of principally polarized complex abelian varieties.
\end{itemize}
\end{myitem}
\end{myitem}

\subsection*{Acknowledgements} The idea for this project was born during the Summer School "Around Moduli Spaces" that took place at Saarland University in September 2019, where Frank Herrlich gave a minicourse on "Schottky groups and moduli spaces". We thank him for his crystal clear lectures and answering several of our questions; we also thank the organizers Gabriele Weitze-Schmidth\"usen, Christian Steinhart, and Andrea Thevis for creating this opportunity. Thanks are also due Martin M\"oller, Filippo Viviani and, in particular, to Annette Werner for many useful discussions; she was part of this project in the beginning, but decided to not be a coauthor. The author thanks Daniele Turchetti and J{\'e}r{\^o}me Poineau for our communication concerning their project \cite{PoineauTurchetti_SchottkyoverZ} and for pointing out a crucial difference between Gerritzen-Herrlich's Schottky space and $\calTbar_g$, as well as Hannah Markwig and Christian Steinhart for keeping him in the loop on an ongoing project on how to tropicalize Archimedean Teichm\"uller space. Finally, we acknowledge support from the LOEWE-Schwerpunkt ``Uniformisierte Strukturen in Arithmetik und Geometrie''.


\section{Tropical moduli stacks}

In this section we introduce the basic terminology that we need to study tropical moduli spaces and introduce the moduli space of tropical curves. Our presentation is distilled from \cite{ACP}, \cite{CCUW}, and \cite{Ulirsch_nonArchArtin}.

\subsection{Cones and cone complexes} An \emph{(abstract) rational polyhedral cone} is a tuple $(\sigma,M)$ consisting of a topological space $\sigma$ together with a finitely generated free subgroup $M$ of the group of continuous real-valued functions on $\sigma$ such that the evaluation map $\sigma\rightarrow \Hom(M,\R)$ induces a homeomorphism between $\sigma$ and a strictly convex rational polyhedral cone in $N_\R:=\Hom(M,\R)$ (i.e. a finite intersection of rational half-spaces). A morphism of rational polyhedral cones $(\sigma,M)\rightarrow (\sigma',M')$ is a continuous map $\sigma\rightarrow \sigma'$ that pulls back $M'$ to $M$. We usually drop the reference to $M$ from our notation. 

The \emph{dual monoid} of a rational polyhedral cone $\sigma$ is the submonoid $S_\sigma$ of those functions in $M$ that are non-negative on $\sigma$. We may recover $\sigma$ from $S_\sigma$ via the identification $\sigma=\Hom(S_\sigma,\R_{\geq 0})$. In fact, the association $(\sigma,M)\mapsto S_\sigma$ defines an equivalence of between the category $\mathbf{RPC}$ of rational polyhedral cones and the category of finitely generated and integral (i.e. fine), saturated, and sharp monoids. A \emph{face} $\tau$ of $\sigma$ is a subset along which a function $u\in S_\sigma$ vanishes. It naturally carries the structure of a rational polyhedral cone and the dual monoid of $\tau$ is given by the quotient $S_u/S_u^\ast$, where $S_u$ denotes the localization 
\begin{equation*}
S_\sigma=\big\{s-ku\big\vert s\in S_\sigma \textrm{ and } k\in\N\big\}
\end{equation*}
and $S_u^\ast$ is the submonoid of units in $S_u$. A \emph{face morphism} is a morphism $\tau\rightarrow\sigma$ that induces an isomorphism between $\tau$ and a face of $\sigma$. Notice here that, in particular, all automorphisms of a rational polyhedral cone are face morphisms. We say that a face morphism is \emph{proper} if it is not an isomorphism. 

\begin{definition}
A \emph{(rational polyhedral) cone complex} is $\Sigma$ is a topological space $\vert\Sigma\vert$ together with a collection of closed subsets $\sigma_i$ (with $i\in I$) that cover $\vert\Sigma\vert$ and carry the structure of a rational polyhedral cone subject to the following axioms:
\begin{enumerate}
\item The intersection $\sigma_i\cap \sigma_j$ is a (necessarily finite) union of faces of each $\sigma_i$ and $\sigma_j$.
\item For every face $\tau$ of $\sigma_i$ there is $j\in I$ such that $\sigma_j=\tau$.
\item A subset $A\subseteq \vert\Sigma\vert$ is closed if and only if $A\cap \sigma_i$ is closed for all $i\in I$. 
\end{enumerate}
\end{definition}

In other words, a cone complex $\Sigma$ is a colimit (in the category of topological spaces) over a poset of cones connected by face morphisms (see \cite[Section 2.1]{CCUW} for details). A morphism $\Sigma\rightarrow\Sigma'$ of cone complexes is a continuous map $\vert\Sigma\vert\rightarrow\vert\Sigma'\vert$ such that for every cone $\sigma_i\subseteq \Sigma$ there is a cone $\sigma_j'\subseteq\Sigma'$ such that $f$ factors through a morphism $\sigma_i\rightarrow\sigma_j'$ in $\mathbf{RPC}$. We denote the category of rational polyhedral cone complexes by $\mathbf{RPCC}$. 

A morphism $f\colon\Sigma\rightarrow \Sigma'$ is said to be \emph{strict}, if the induced map $\sigma_i\rightarrow \sigma_j'$ is a face morphism. Denote the class of strict morphism by $\mathbb{P}_{strict}$. By \cite{CCUW}, strict morphisms define a subcanonical Grothendieck topology $\tau_{strict}$ on $\mathbf{RPCC}$, and the triple $\big(\mathbf{RPCC},\tau_{strict}, \mathbb{P}_{strict}\big)$ defines a geometric context in the sense of \cite[Section 1]{CCUW}. 

\begin{remark}
The category $\mathbf{RPCC}$ is naturally equivalent to the category of (fine and saturated) Kato fans, an incarnation of the geometry over the field with one element introduced in \cite{Kato_toricsing} (see \cite[Prop. 3.7]{Ulirsch_functroplogsch}). Strict morphisms hereby correspond exactly to local isomorphism of Kato fans. 
\end{remark}


\subsection{Cone stacks and combinatorial cone stacks}

Let $\Sigma$ be a cone complex. We identify $\Sigma$ with its associated functor of points 
\begin{equation*}\begin{split}
h_\Sigma\colon \mathbf{RPCC}&\longrightarrow \mathbf{Sets}\\
\Sigma'&\longmapsto \Hom(\Sigma',\Sigma)
\end{split}\end{equation*}
so that, by Yoneda's Lemma, we can think of $\Sigma$ as both a presheaf and as a category fibered in groupoids $\mathbf{RPCC}/\Sigma$ over $\mathbf{RPCC}$. We, in particular, say that a category fibered in groupoids over $\mathbf{RPCC}$ is \emph{representable} by $\Sigma$ if it is equivalent to $\mathbf{RPCC}/\Sigma$.

\begin{definition}
A \emph{(rational polyhedral) cone stack} is a category $\calC$ fibered in groupoids over $\mathbf{RPCC}$ that is a stack with respect to the strict topology $\tau_{strict}$ that fulfils the following two axioms:
\begin{enumerate}[(i)]
\item the diagonal $\Delta\colon \calC\longrightarrow \calC\times\calC$ is representable by cone complexes; and 
\item there is a cone complex $\Sigma$ and a (necessarily representable) morphism $\Sigma\rightarrow \calC$ that is strict and surjective. 
\end{enumerate}
\end{definition}

The map $\Sigma\rightarrow \calC$ is also called a chart of $\calC$. Cone stacks naturally form a $2$-category. Its morphisms are morphism of categories over $\mathbf{RPCC}$. The usual techniques for working with stacks apply to this situation as well. In particular, given a strict surjective groupoid object $R\rightrightarrows U$ in $\mathbf{RPCC}$ the quotient $\big[U\big/R\big]$ is cone stack, and, conversely, given a chart $U\rightarrow \calC$ of a cone stack $\calC$, the fiber product $U\times_\calC U$ is representable by a cone complex $R$ so that $R\rightrightarrows U$ defines a strict and surjective groupoid object in $\mathbf{RPCC}$ and we have a natural equivalence $\big[U/R\big]\simeq \calC$. 


In \cite[Section 2.2]{CCUW} the authors have introduced a combinatorial characterization of cone stacks. 

\begin{definition}
A \emph{combinatorial cone stack} is a category fibered in groupoids over $\mathbf{RPC}^f$, the category of rational polyhedral cones with only face morphisms. 
\end{definition}

By \cite{CCUW}, there is a natural equivalence between the $2$-category of cone stacks and the $2$-category of combinatorial cone stacks. Given a cone stack $\calC$ an object in the associated combinatorial cone stack is a strict morphism $\sigma\rightarrow \calC$ from a rational polyhedral cone $\sigma$ into $\calC$ and a morphism is a commuting diagram
\begin{center}\begin{tikzcd}
\sigma \arrow[rr]\arrow[rd] & & \sigma'\arrow[ld]\\
& \calC & 
\end{tikzcd}\end{center} 
which is automatically a face morphism. Conversely, given a combinatorial cone stack $\calC^{comb}$, the associated cone stack $\calC$ is the unique stack over $(\mathbf{RPCC},\tau_{strict})$ whose fiber over a cone $\sigma$ is the groupoid $\HOM(\sigma,\calC^{comb})$. 


\subsection{Coarse moduli spaces and generalized cone complexes}\label{section_coarsemodulispaces}
In the following we write $\FAN$, when we think of $\mathbf{RPC}^{f}$ as a category fibered in groupoids over itself. So for every cone stack $\calC$, there is a tautological morphism $\calC^{comb}\rightarrow \FAN$. 

\begin{definition}
A cone stack $\calC$ is said to have \emph{faithful monodromy}, if the tautological morphism $\calC^{comb}\rightarrow \FAN$ is representable. 
\end{definition}

In other words, this means that morphisms in $\calC^{comb}$ are all actual face morphism. Alternatively, one may also think of cone stacks with faithful monodromy as (relative) sheaves over the category $\FAN$. 

\begin{proposition}\label{prop_coarsemodulispace}
Let $\calC$ be a cone stack. Then there is a cone stack $C$ with faithful monodromy together with a strict morphism $\calC\rightarrow C$ that is initial among all strict morphisms from $\calC$ into cone stacks with faithful monodromy. 
\end{proposition}

In other words $\calC\rightarrow C\rightarrow \FAN$ is the initial factorization of the tautological morphisms $\calC\rightarrow \FAN$ such that $C\rightarrow \FAN$ is representable. So the $\calC\rightarrow C$ is the relative coarse moduli space of $\calC$ over $\FAN$. In a way, the morphism $\calC\rightarrow C$ plays the role of the morphism of an algebraic stack to its coarse moduli space. Therefore we refer to $\calC\rightarrow C$ and, in a slight abuse of notation, to $C$ as the \emph{coarse moduli space} of $\calC$. 

\begin{proof}[Proof of Proposition \ref{prop_coarsemodulispace}]
Using the identification of cone stacks with Artin fans from \cite[Theorem 3]{CCUW} (see also Section \ref{section_Artinfans} below), Proposition \ref{prop_coarsemodulispace} is a special case of \cite[Propostion 3.1.1]{ACMW}. The combinatorial cone stack $C^{comb}$ has the same objects as $\calC$. For two objects $\alpha$ and $\beta$ in $\calC^{comb}$ over cones $\sigma$ and $\tau$ respectively, the morphisms in $C^{comb}$ are the image of $\Hom_\calC(\alpha, \beta)$ in $\Hom(\sigma,\tau)$. We may now easily verify that $C^{comb}$ fulfils the axioms of a category fibered in groupoids over $\mathbf{RPC}^{f}$, that it has faithful monodromy, and that $\calC^{comb}\rightarrow C^{comb}$ is initial among all strict morphisms to combinatorial cone stacks with faithful monodromy.
\end{proof}

Cone stacks are a refinement of the notion of a \emph{generalized cone complexes}, as introduced in \cite{ACP} as a generalization of cone complexes. 

\begin{definition}[\cite{ACP}] A \emph{generalized cone complex} as a topological space together with a presentation as a colimit of a diagram of (not necessarily proper) face morphisms. 
\end{definition}

A morphism of generalized cone complexes is a continuous map that locally factors through a morphism in $\mathbf{RPC}$. The combinatorial cone stack associated to a cone stack with faithful monodromy defines a generalized cone complex. Conversely, adding all faces and pullback of invariant automorphisms, the defining diagram of a generalized cone complex generates a combinatorial cone stack. In fact, we have a natural equivalence (of $1$-categories)
\begin{equation*}
\big\{\textrm{cone stacks with faithful monodromy}\big\}\big/_{\big\{2\textrm{-isomorphisms}\big\}}\simeq \big\{\textrm{generalized cone complexes}\big\} \ .
\end{equation*}


\subsection{Graphs}

Expanding on \cite{Serre_trees} and \cite{CCUW}, a \emph{graph} $G$ consists of a set $X=X(G)$ together with an idempotent \emph{root map} $r\colon X\rightarrow X$ and an involution $i\colon X\rightarrow X$ such that $r\circ i=i\circ r$. We refer to the set $V(G)=r(X)$ as the set of \emph{vertices} of $G$ and to its complement as the set $H(G)$ of half edges of $G$. An element in the quotient $H(G)/i$ is of the form $\big[h\sim i(h)\big]$ for an half-edge $h$ of $G$; we refer to $\big[h\sim i(h)\big]$ as a \emph{finite edge} when $h\neq i(h)$ and otherwise as a \emph{leg}. So the quotient $X/i$ decomposes as a disjoint union $V(G)\sqcup E(G)\sqcup L(G)$, where $E(G)$ is the set of finite edges and $L(G)$ is the set of legs. 

We say that a graph $G$ is \emph{finite}, if $X(G)$ is finite. An \emph{order} on an edge $e=\big[h\sim i(h)\big]$ is the choice of a relation $h<i(h)$ or $h>i(h)$. For an ordered edge $e=\big[h<i(h)\big]$, we write $\overline{e}$ for the same edge with the reverse order $\big[h>i(h)\big]$.  A \emph{path} $\gamma$ in a graph $G$ is a tuple $(e_1,\ldots, e_n)$ of ordered edges $e_i$ of $G$ such that for every ordered edge $e_i=\big[h_i< \widetilde{h}_i\big]$ (with $1\leq i\leq n$) we have $r(h_i)=r(\widetilde{h}_{i+1})$. We write $v_i$ for the vertex $r(h_i)=r(i(h_{i+1}))$ as well as $v_0=r(i(h_1))$ and $v_n=r(h_n)$.  We say that $G$ is \emph{connected}, if for any two vertices $v,w$ there is a path $\gamma$ in $G$ with $v_0=v$ and $v_n=w$.  A path is said to be \emph{closed} if $v_0=v_n$; in this case we refer to $v_0=v_n$ as the \emph{base point} of the path $\gamma$. From now on we assume that all our graphs are connected.

A \emph{vertex weight} on a graph $G$ is a function $h\colon V(G)\rightarrow \Z_{\geq 0}$; a \emph{marking} $m$ on the set of legs of $G$ is a choice of total order on $L$. Whenever convenient we drop the reference to $h$ and $m$ from our notation and denote a weighted (marked) graph simply by $G$. 

The \emph{valence} $\val(v)$ of a vertex $v$ of $G$ is the number of half edges $h$ with $r(f)=v$. A weighted marked graph is said to be \emph{stable}, if for all vertices $v\in V(G)$ we have 
\begin{equation*}
2h(v)-2+\val(v)>0 \ . 
\end{equation*}
The \emph{genus} $g(G)$ of $G$ is defined to be $b_1(G)+\sum_{v\in V}h(v)$.

Let $G, G'$ be two weighted marked graphs. A \emph{(generalized) weighted edge contraction} is a map $\pi\colon X\rightarrow X'$ that fulfils the following axioms:
\begin{itemize}
\item $\pi$ commutes with commutes with $r$, $i$, and $h$; 
\item the preimage $\pi^{-1}(f')$ of each half-edge $h'\in H(G')$ consists of precisely one element $f$ (which is necessarily a half-edge of $G$);
\item $\pi$ induces an order preserving bijection $L\xrightarrow{\sim} L'$; and
\item for every $v'\in V(G')$ the preimage $\pi^{-1}(v)$ is a connected finite weighted graph of genus $h(v')$. 
\end{itemize}
We denote by $J_{g,n}$ the category, whose objects are finite weighted stable graph $G$ of genus $g$ with $n$ marked legs, and whose morphisms are weighted edge contractions. 


\subsection{Tropical curves}

\begin{definition}
Let $P$ be a monoid. A \emph{tropical curve} $\Gamma$ over $P$ is a finite weighted graph $G(\Gamma)=(G,h,m)$ together with a \emph{generalized edge length} $\vert .\vert \colon E(G)\rightarrow P-\{0\}$. 
\end{definition}

A tropical curve $\Gamma$ (over $P$) is said to be \emph{stable}, if $G(\Gamma)$ is stable. The \emph{genus} $g(\Gamma)$ of a tropical curve is the genus of the underlying weighted graph $G(\Gamma)$. A \emph{(generalized) weighted edge contraction} $\pi\colon \Gamma\rightarrow \Gamma'$ of tropical curves $\Gamma$ over $P$ and $\Gamma'$ over $P'$ consists of a monoid homomorphism $\pi^\flat\colon P\rightarrow P'$ and a weighted edge contraction $\pi\colon G(\Gamma)\rightarrow G(\Gamma')$ such that
\begin{itemize}
\item $\pi$ contracts an edge if and only $\pi^\flat(\vert e\vert)=0$ and
\item if $\pi(e)=e\in H(G')$, then $\pi^\flat(\vert e\vert) =\vert \pi(e)\vert$. 
\end{itemize}

 Let $g,n\geq 0$ such that $2g-2+n>0$. By \cite[Proposition 2.3]{CCUW}, there is a unique stack $\calM_{g,n}^{trop}$ over $(\mathbf{RPCC},\tau_{strict})$, whose fiber over a cone $\sigma$ is the groupoid of stable tropical tropical curves of genus $g$ with $n$ marked legs. We refer to $\calM_{g,n}^{trop}$ as the \emph{moduli stack of tropical curves (of genus $g$ with $n$ marked points)}. 

\begin{theorem}[\cite{CCUW} Theorem 1]
The stack $\calM_{g,n}^{trop}$ is a cone stack. 
\end{theorem}

In fact, one way to prove this, is to realize that $\calM_{g,n}^{trop}$ is the cone stack associated to the combinatorial cone stack, defined by the functor
\begin{equation*}\begin{split}
J_{g,n}^{op}&\longrightarrow\mathbf{RPC}^f\\
G&\longmapsto \sigma_G=\R_{\geq 0}^E(G) \ ,
\end{split}\end{equation*}
where a weighted edge contraction $G\rightarrow G'$ naturally induces a face morphism $\sigma_{G'}\rightarrow \sigma_G$. 

The moduli stack $\calM_{g,n}^{trop}$ does not have faithful monodromy, since there are non-trivial automorphisms of graphs that only induces a trivial permutation of the set of edges. Nevertheless, the image of $J_{g,n}$ in $\mathbf{RPC}^f$ has faithful monodromy and, by \cite[Theorem 1.3]{Ulirsch_nonArchArtin} the resulting cone stack (with faithful monodromy) functions as a coarse moduli space for $\calM_{g,n}^{trop}$, in the sense that it is initial among all strict morphisms $\calM_{g,n}^{trop}\rightarrow \calC$ to cone stacks with faithful monodromy. In a slight abuse of notation we denote by $M_{g,n}^{trop}$ both the coarse moduli space of $\calM_{g,n}^{trop}$ and the associated generalized cone complex. 


\section{Uniformization in the tropics}
In this section we first recall from \cite{Bass_graphsofgroups, Serre_trees} the theory of \emph{graphs of groups} and their fundamental groups. We then use these techniques to expand on \cite{CMV} and construct tropical Teichm\"uller space $\calT_g^{trop}$ as a cone stack that is representable by a cone complex $T_g^{trop}$.

\subsection{Graphs of groups} 

\begin{definition}
A \emph{graph of groups} $\G$ is a graph $G=(V,E,L)$ together with 
\begin{itemize}
\item a group $G_v$ for every vertex $v$ of $G$;
\item a group $G_f$ for every half-edge $f$ of $G$ together with an isomorphism $G_f\xrightarrow{\sim} G_{i(f)}$ denoted by $g\mapsto \overline{g}$; and
\item monomorphisms $G_f\rightarrow G_{r(f)}$ for every half-edge $f$ of $G$ denoted by $g\mapsto a^g$. 
\end{itemize}
The group $G_v$ is called the \emph{vertex group} of the vertex $v\in V(G)$ and $G_f$ the \emph{edge group} of $f\in F(G)$. 
\end{definition}


A \emph{word} in a graph of groups $\G$ is a pair $(\gamma,\vec{g})$ consisting of a path $\gamma=(e_1,\ldots, e_n)$ in $G$ consisting of ordered edges $e_i=\big[f_i< \widetilde{f}_i\big]$ connecting $v_{i-1}$ to $v_i$ (with $1\leq i\leq n$) and a tuple $\vec{g}=(g_0,\ldots, g_n)$ of elements $g_i\in G_{v_i}$. Let $(\gamma,\vec{g})$ and $(\gamma',\vec{g'})$ be two words in $\G$ such that the path $\gamma'$ starts at the end point of $\gamma$, i.e. for which we have $v_0=v_{n+1}$. The \emph{concatenation} of $(\gamma,\vec{g})$ and $(\gamma',\vec{g'})$ is given by the concatenation of $\gamma\circ \gamma'$ and the tuple 
\begin{equation*}
\vec{g}\circ\vec{g}'=(g_0,\ldots, g_{n-1},g_n\cdot g_0',g_1',\ldots, g'_{n'}) \ .
\end{equation*}
Concatenation of words is associative and, writing  $(v,1_{G_v})$ for the trival word at the vertex $v$, we have $(v_0,1_{G_{v_0}})\circ (\gamma,\vec{g})=(\gamma,\vec{g})=(\gamma, \vec{g})\circ (v_n,1_{G_{v_n}})$. Moreover, for every word $(\gamma,\vec{g})$, there is an \emph{inverse word} $(\gamma^{-1},\vec{g}^{-1})$ given by the inverse $\gamma^{-1}=(e_n^{-1},\ldots, e_1^{-1})$ of $\gamma$ and the vector $\vec{g}^{-1}=(g_n^{-1},\ldots, g_0^{-1})$. 


\begin{definition}
Let $\G$ be a graph of groups and $v\in V(G)$. The \emph{fundamental groupoid} $\pi_1(\G)$ of $\G$ is the groupoid whose objects are the vertices of $G$ and whose morphisms are generated by the words in $\G$ subject to the relations
\begin{equation*}
\overline{e}=e^{-1} \quad\textrm{and}\quad e a^g e^{-1} =a^{\overline{g}}  
\end{equation*}
for all oriented edges $e=\big[f<\widetilde{f}\big]$ and $g\in G_f$. For a vertex $v\in V(G)$ the \emph{fundamental group} $\pi_1(\G,v)$ of $\G$ based at $v$ is the group of automorphism of $v$ in $\pi_1(\G)$. 
\end{definition}

Suppose that $G$ is a graph and $\G=(G,1)$ is the trivial graph of groups on $G$, i.e. the groups $G_v$ and $G_h$ are all trivial. Then $\pi_1(\G)=\pi_1(G)$, the classical fundamental groupoid of the graph $G$ (defined combinatorially). For every graph of groups $\G$ with underlying graph $G$, the canonical morphism $\G\rightarrow (G,1)$ induces a surjective homomorphism $\pi_1(\G)\twoheadrightarrow \pi_1(G)$; its kernel is the normal subgroupoid generated by all the $G_v$. 

Let $w\in V(G)$ be another base point of $G$. Let $\gamma$ be a path connecting $v$ to $w$. Then there is a natural isomorphism
\begin{equation*}\begin{split}
\pi_1(\G,v)&\xlongrightarrow{\sim}\pi_1(\G,w)\\
a&\longmapsto \gamma^{-1}a\gamma \ .
\end{split}\end{equation*}


\subsection{From weighted graphs to graphs of groups}
To a weighted graph $(G,h)$ we associate a graph of groups $\G(G,h)$ with underlying graph $G$ by endowing every vertex with the free group $F_{h(v)}$ on $h(v)$ generators and every half-edge $f$ with the trivial group (together with the unique monomorphisms $G_f={1}\hookrightarrow G_r(f)$). We then set
\begin{equation*}
\pi_1(G,h)=\pi_1\big(\G(G,h)\big) \qquad 
\end{equation*} 
as well as 
\begin{equation*}
\pi_1(G,h;v)=\pi_1\big(\G(G,h),v\big) 
\end{equation*}
for a base point $ v\in V(G)$ and refer to $\pi_1(G,h)$ (and $\pi_1(G,h;v)$) as the \emph{fundamental groupoid} $\pi_1(G,h)$ (respectively the \emph{fundamental group} $\pi_1(G,h; v)$ with base point $v$).

\begin{proposition}\label{prop_weightededgecontractions}
For a weighted edge contraction $\phi\colon(G',h')\rightarrow (G,h)$ there is an equivalence \begin{equation*}
\pi_1\big(\G(G',h')\big)\simeq\pi_1\big(\G(G,h)\big)\ .
\end{equation*}
\end{proposition}

\begin{proof}
Let $\phi\colon(G',h')\rightarrow (G,h)$ be a weighted edge contraction. Choose base points $v'\in V(\pi^{-1}(v))$ for every $v\in V(G')$ and denote by $\G(\pi)$ the graph of groups (with underlying graph $G$) whose group at a vertex $v\in V(G)$ is given by $\pi_1\big(\G(\pi^{-1}(v));v'\big)$ (and trivial groups along all edges). The choice of an isomorphism $\pi_1\big(\G(\pi^{-1}(v));v'\big)\simeq F_{h(v)}$ induces the desired equivalence $\pi_1(G',h')\simeq\pi_1\big(\G(G,h)\big)$.
\end{proof}

\begin{corollary}
For a finite weighted graph $(G,h)$ of genus $g$ there is an equivalence
\begin{equation*}
\pi_1\big(\G(G,h)\big)\simeq F_g \ .
\end{equation*}
\end{corollary}

\begin{proof}
Apply Proposition \ref{prop_weightededgecontractions} to the weighted edge contraction $(G,h)\rightarrow (\ast, g)$ that contract $G$ to a point with vertex weight $g=g(G,h)$. 
\end{proof}

This allows us to define the following.

\begin{definition}
Let $(G,h)$ be a finite weighted graph. A \emph{Teichm\"uller marking} on $(G,h)$ is an equivalence 
\begin{equation*}\phi\colon \pi_1(G,h)\xlongrightarrow{\sim} F_g \ .
\end{equation*} 
\end{definition}

In other words, a Teichm\"uller marking is an outer isomorphism class $\phi_v\colon \pi_1(\G,v)\xrightarrow{\sim}F_g$ for one (and, since $G$ is connected, all) $v\in V(G)$. 

\begin{definition}
Let $(G,h)$ be a finite weighted graph of genus $g$. Two equivalences $\phi_i\colon\pi_1(G,h)\xrightarrow{\sim}F_g$ (for $i=1,2$) are said to be \emph{topologically equivalent}, if for one (and therefore all) $v\in V(G)$ the induced surjective homomorphisms 
\begin{equation*}
F_g\xlongrightarrow{\phi_i}\pi_1(G,h;v)\twoheadlongrightarrow \pi_1(G,v)
\end{equation*}
for $i=1,2$ are equal.
\end{definition}

Topological equivalence defines an equivalence relation on the class of all Teichm\"uller markings; we write $\big[\phi\colon \G(\Gamma)\xrightarrow{\sim}F_g\big]$ for the topological equivalence class associated to a Teichm\"uller marking.

\subsection{Tropical Teichm\"uller space}

Given a tropical curve $\Gamma$ (over a monoid $P$), we write $\G(\Gamma)$ for the graph of groups associated to the underlying weighted graph $(G,h)$ of $\Gamma$. Moreover, we denote $\pi_1(\Gamma)=\pi_1\big(\G(\Gamma)\big)$ and $\pi_1(\Gamma,v)=\pi_1\big(\G(\Gamma),v\big)$ for $v\in V(G)$. A \emph{Teichm\"uller marking} on a tropical curve $\Gamma$ is a Teichm\"uller marking on the underlying finite weighted graph $G(\Gamma)$. 

Let $g\geq 2$. By \cite[Proposition 2.3]{CCUW} there is a unique stack $\calT_{g}^{trop}$ over $\big(\mathbf{RPCC},\tau_{strict}\big)$ whose fiber over a rational polyhedral cone $\sigma$ is the groupoid of pairs consisting of a stable tropical curve $\Gamma$ of genus $g$ together with a topological equivalence class of Teichm\"uller marking $\big[\phi\colon \G(\Gamma)\xrightarrow{\sim}F_g\big]$. 

\begin{theorem}\label{thm_tropTeichmueller}
The space $\calT_{g}^{trop}$ is representable by a cone complex $T_{g,n}^{trop}$. 
\end{theorem}

Following \cite{CMV} we introduce the following terminology.

\begin{definition}
The cone complex $T_{g}^{trop}$ is called \emph{tropical Teichm\"uller space}.
\end{definition}

\begin{proof}[Proof of Theorem \ref{thm_tropTeichmueller}]
Consider the category $\widetilde{J}_{g}$ whose objects are tuples $\big(G,h,[\phi]\big)$ consisting of a vertex-weighted graph $(G,h)$ of genus $g$ and a topological equivalence class of a Teichm\"uller marking $\big[\phi\colon \G(G,h)\xrightarrow{\sim}F_g\big]$ and
whose morphism are weighted edge contractions. We note that, for a weighted edge contraction $(G,h)\rightarrow(G',h')$, a Teichm\"uller marking $\phi\colon\G(G,h)\xrightarrow{\sim}F_g$ naturally induces a Teichm\"uller marking of $(G',h')$ by Proposition \ref{prop_weightededgecontractions}. The natural function $\widetilde{J}_{g}\rightarrow \mathbf{RPC}^{f}$ given by $\big(G,h,[\phi,]\big)\mapsto \sigma_{(G,h,[\phi])}=\R_{\geq 0}^{E(G)}$ makes $\widetilde{J}_{g}$ into a category fibered in groupoids over $\mathbf{RPC}^{f}$, i.e. into a combinatorial cone stack. The associated cone stack is equivalent to $\calT_{g}^{trop}$, since strict morphisms $\sigma\rightarrow\calT_{g}^{trop}$ from rational polyhedral cones naturally correspond to objects in $\widetilde{J}_{g}$. 

We now show that $\widetilde{J}_{g}\rightarrow \mathbf{RPC}^{f}$ defines a cone complex: The operation of the automorphism group of a finite graph on its fundamental groupoid is faithful. Therefore the automorphism group of $\big(G,h,[\phi]\big)$ in $\widetilde{J}_{g}$ is trivial and thus $\widetilde{J}_{g}\rightarrow \mathbf{RPC}^{f}$ is fibered in sets (and not groupoids). It is a poset, since, whenever we have two weighted edge contraction $(G,h)\rightrightarrows (G',h')$, there already is an automorphism of $(G,h)$ that makes the diagram
\begin{center}\begin{tikzcd}
(G,h)\arrow[rr,"\simeq"]\arrow[rd]&&(G,h)\arrow[ld]\\
& (G',h') &
\end{tikzcd}\end{center}
commute. Thus the colimit of the diagram $\widetilde{J}_{g}\rightarrow \mathbf{RPC}^{f}$ is a cone complex $T_{g}^{trop}$.
\end{proof}

\begin{remark}
The locus of pairs $\big(\Gamma, [\phi\big]\big)$ in $T_g^{trop}$ where the vertex weight function is trivial is precisely the equal to space of metric graphs together with a Teichm\"uller marking $\phi\colon \pi_1(\Gamma)\xrightarrow{\sim} F_g$ (without reference to topological equivalence, since all vertex groups of $\G(\Gamma)$ are trivial). As explained in \cite[Section 3.2]{CMV} this space is naturally homeomorphic to the (not volume-normalized) Outer space in the sense of Culler and Vogtmann \cite{CullerVogtmann}. In \cite{CullerVogtmann} the authors impose the that for metric graphs $\Gamma$ in $\CV_g$ the condition 
\begin{equation*}
\sum_{e\in E(\Gamma)}\vert e\vert =1
\end{equation*}
on the total length of $\Gamma$ holds. As in \cite{CMV}, we do not follow this convention.
\end{remark}

\subsection{The quotient by $\Out(F_g)$} There is natural operation of the group $\Aut(F_g)$ on $\calT_{g}^{trop}$ that is given by sending $\big(\Gamma,[\phi]\big)$ to $\big(\Gamma, [g\circ \phi]\big)$ for $g\in \Aut(F_g)$. An equivalence $\phi\colon \pi_1(\Gamma)\xrightarrow{\sim}F_g$ is determined only up to inner automorphisms of $F_g$ and so the group $\Inn(F_g)$ of inner automorphisms of $F_g$ acts trivially on $\calT_g^{trop}$. Thus there is a natural induced operation of $\Out(F_g)=\Aut(F_g)/\Inn(F_g)$ on $\calT_g^{trop}$.

Consider now the natural morphism $\calT_g^{trop}\rightarrow \calM_g^{trop}$ that is given by forgetting the Teichm\"uller marking. Since $\Out(F_g)$ only operates on the markings, the map $\calT_g^{trop}\rightarrow \calM_g^{trop}$ is invariant under this operation and there is an induced morphism $\big[\calT_g^{trop}\big/\Out(F_g)\big]\rightarrow \calM_g^{trop}$. 

\begin{theorem}\label{thm_tropquot}
The relative coarse moduli space of $\big[\calT_g^{trop}\big/\Out(F_g)\big]$ over $\calM_g^{trop}$ is naturally equivalent to $\calM_g^{trop}$. 
\end{theorem}

\begin{proof}
The induced morphism $\big[\calT_g^{trop}\big/\Out(F_g)\big]\rightarrow \calM_g^{trop}$ is essentially surjective, since every tropical curve $\Gamma$ can be endowed with a Teichm\"uller marking $\phi\colon \pi_1(\Gamma)\xrightarrow{\sim}F_{g}$. It is full, since every weighted edge contraction $\phi\colon\Gamma\rightarrow\Gamma'$ induces an equivalence $\pi_1(\Gamma)\xrightarrow{\sim} \pi_1(\Gamma')$ by Proposition \ref{prop_weightededgecontractions}. Going from $\big[\calT_g^{trop}\big/\Out(F_g)\big]$ to the relative coarse moduli space over $\calM_g^{trop}$ makes the induced map also faithful and thus the result follows. 
\end{proof}


\section{Lifting via Artin fans}

In this section we use methods from logarithmic geometry in the sense of Kato-Fontaine-Illusie \cite{Kato_logstr} and, in particular, the theory of Artin fans (as in \cite{AW, ACMW,ACMUW, Ulirsch_nonArchArtin,CCUW}), to lift tropical Teichm\"uller space to the world of algebraic geometry and to study the process of tropicalization.

\subsection{Logarithmic structures} Recall from \cite{Kato_logstr} that a logarithmic structure on a scheme $\underline{X}$ is a pair $(M_X,\alpha_X)$ consisting of 
\begin{itemize}
\item a sheaf of monoids $M_X$ defined on the \'etale topology on $\underline{X}$, and 
\item a monoid homomorphism $\alpha_X\colon M_X\rightarrow (\calO_X,\cdot)$ that induces an isomorphism $\alpha_{X}^{-1}\calO_{\underline{X}}^\ast\simeq \calO_{\underline{X}}^\ast$. 
\end{itemize}
We refer to the tuple $X=(\underline{X},M_X,\alpha_X)$ consisting of a scheme $\underline{X}$ and a logarithmic structure $(M_X,\alpha_X)$ as a \emph{logarithmic scheme}. Whenever convenient we drop the reference to $\alpha_X$ and simply write $X=(\underline{X},\alpha_X)$ for a logarithmic scheme. We write $\Mbar_X$ for the quotient $\Mbar_X=M_X/M_{X}^\ast$, which is known as the \emph{characteristic monoid} of $X$. 

A logarithmic scheme $X$ is called \emph{fine and saturated} if \'etale locally there is a homomorphism $P_{\underline{X}}\rightarrow (\calO_{\underline{X}},\cdot)$ from the constant sheaf $P_{\underline{X}}$ associated to a fine and saturated monoid $P$ to $(\calO_{\underline{X}},\cdot)$ such that the logarithmic structure $M_X$ is given via the pushout square
\begin{center}\begin{tikzcd}
P_{\underline{X}}^\ast\arrow[rr] \arrow[d,"\subseteq"]&& \calO_{\underline{X}}^{\ast}\arrow[d]\\
P_{\underline{X}}\arrow[rr] && M_{X}
\end{tikzcd}\end{center}

For further details on logarithmic geometry we refer the avid reader to \cite{Kato_logstr}, \cite{Abramovichetal_logmoduli}, and \cite{Ogus_logbook}. From now on the terms \emph{logarithmic scheme} or \emph{logarithmic stack} will always refer to a fine and saturated logarithmic scheme or logarithmic stack. We denote the category of (fine and saturated) logarithmic scheme by $\mathbf{LSch}$ and the category of (fine and saturated) logarithmic stacks by $\mathbf{LSt}$. 

\subsection{Logarithmic curves} 
A \emph{logarithmic curve} over a logarithmic base scheme $S$ is a logarithmically smooth morphism $X\rightarrow S$ that is proper, integral, saturated, and has geometrically connected fibers of dimension one. 

\begin{theorem}[\cite{Kato_logsmoothcurves} Theorem 3.1] \label{thm_logcurves}
Let $ X\rightarrow S$ be a logarithmic curve. Then every point $x$ of $X$ has an \'etale neighborhood $V$ together with a morphism $\pi\colon V\rightarrow S$ such that one of the following holds:
\begin{enumerate}[(i)]
\item $V=\Spec \calO_S[u]$ with $M_V=\pi^\ast M_S$; 
\item $V=\Spec \calO_S[u]$ with $M_V=\pi^\ast M_S\oplus \N v$ with $\alpha_V(v)=u$; or 
\item $V=\Spec \calO_S[x,y]\big/(xy-t)$ for some $t\in\calO_S$ and 
\begin{equation*}
M_V=\pi^\ast M_S\oplus \N\alpha \oplus \N\beta\big/(\alpha +\beta=\delta)
\end{equation*}
for some $\delta\in \pi^\ast M_S$ with $\epsilon_V(\alpha)=x$, $\epsilon_V(\beta)=y$, and $\epsilon_S(\delta)=t$. 
\end{enumerate}
\end{theorem}

So the underlying family of curves $\overline{X}\rightarrow \overline{S}$ is flat and proper, and the each fiber is nodal curve with a finite number of (a priori unordered) sections that do not meet the singularities in each fiber. We define the \emph{moduli stack $\calM_{g,n}^{log}$ of logarithmic curves} to be the unique stack over $\mathbf{LSch}$ whose fiber over a logarithmic base scheme $S$ is the groupoid of stable logarithmic curves of genus $g$ with $n$ marked sections. 

The connection with the classical Deligne-Knudsen-Mumford moduli stack $\calMbar_{g,n}$ is established by the following:

\begin{theorem}[\cite{Kato_logsmoothcurves} Theorem 4.5]\label{thm_Katologsmoothcurves}
The moduli stack is represented by the pair $(\calMbar_{g,n},M_{g,n})$ where $M_{\calMbar_{g,n}}$ is the divisorial logarithmic structure associated to the boundary divisor of $\calMbar_{g,n}$, i.e. the complement of the locus $\calM_{g,n}$ of smooth $n$-marked curves of genus $g$ in $\calMbar_{g,n}$. 
\end{theorem}

\subsection{From cone stacks to Artin fans}\label{section_Artinfans}

Let $S$ be a logarithmic scheme and $\sigma$ be a rational polyhedral cone. Denote by $\calA_\sigma$ the quotient stack
\begin{equation*}
\calA_\sigma=\big[\Spec \Z[S_\sigma]\big/\Spec \Z[S_\sigma^{gp}]\big]
\end{equation*}
of the affine toric variety $\Spec \Z[S_\sigma]$ by the diagonalizable group $\Spec\Z[S_\sigma^{gp}]$. By \cite[Proposition 5.17]{Olsson_LOG}, for every logarithmic scheme $X$ there is a natural isomorphism
\begin{equation*}
\Hom_{\mathbf{LSt}}\big(X,\calA_\sigma\big)=\Hom_{\mathbf{Mon}}\big(S_\sigma,\Mbar_{X}\big) \ .
\end{equation*}
As in \cite{Ulirsch_nonArchArtin} this observation implies that the association $\sigma\mapsto \calA_\sigma$ defines full and faithful functor from $\mathbf{RPC}$ to the category of logarithmic stacks. We refer to a logarithmic stack of the form $\calA_\sigma$ as an \emph{Artin cone}. 

\begin{definition}
An \emph{Artin fan} is a logarithmic algebraic stack that admits a cover by a disjoint union of Artin cones that is strict and \'etale. 
\end{definition}

In \cite{CCUW} we have seen the following:

\begin{theorem}[\cite{CCUW} Theorem 3]\label{thm_Artinfans=conestacks}
The category of Artin fans is naturally equivalent to the category of cone stacks. 
\end{theorem}

When $\tau$ is a face of a rational polyhedral cone $\sigma$, the induced homomorphism $S_\sigma\rightarrow S_\tau$ determines an open immersion $\calA_\tau\subseteq\calA_\sigma$. So, if $\Sigma$ is a rational polyhedral cone complex, then we may construct the associated Artin fan $\calA_\Sigma$ as 
\begin{equation*}
\calA_\Sigma=\bigcup_{\sigma\subseteq\Sigma} \calA_\sigma \ .
\end{equation*}
In general, given a cone stack $\calC$, we may choose a strict groupoid presentation $\big[U\big/R\big]\simeq \calC$ in $\mathbf{RPCC}$ and construct the associated Artin fan $\calA_\calC$ as the quotient of the induced strict \'etale groupoid
\begin{equation*}
\calA_R\rightrightarrows \calA_U \ . 
\end{equation*}

\begin{remark}
In \cite{CCUW} the proof of Theorem \ref{thm_Artinfans=conestacks} is only written for logarithmic schemes over a field $k$. It, however, directly generalizes to logarithmic schemes over $\Z$ (and in fact to any other logarithmic base scheme with trivial logarithnmic structure).
\end{remark}

In order to keep our notation less bulky, we usually denote both the cone stack $\calC$ and the associated Artin fan $\calA_\calC$ with the same letter $\calC$. 


\subsection{Construction of $\calT_g^{log}$}\label{section_logTeichmueller}
Let $g,n\geq 0$ such that $2g-2+n>0$. By \cite[Theorem 4]{CCUW} the Artin fan associated to tropical moduli stack $\calM_{g,n}^{trop}$ is the category whose fiber over a logarithmic scheme $S$ is the groupoid of families of tropical curves over $S$. A \emph{family of tropical curves over $S$} consists of
\begin{itemize}
\item a collection $(\Gamma_s)$ of tropical curves $\Gamma_s\in\calM_{g,n}^{trop}(\Mbar_{S})$ with edge lengths in the characteristic monoid $\Mbar_S$ indexed by all geometric points $s$ of $S$; and
\item for every \'etale specialization $t\leadsto s$ of geometric points of $S$ a weighted edge contraction $\Gamma_s\rightarrow\Gamma_t$ such that, whenever $\Gamma_s$ is metrized via the composition $\vert .\vert_{t\leadsto s}\colon E(\Gamma)\rightarrow \Mbar_{S,s}\rightarrow \Mbar_{S,t}$, the tropical curve $\Gamma_t$ is given by contracting those edges $e$ in $\Gamma_s$ for which $\vert e\vert_{t\leadsto s}=0$. 
\end{itemize}
Again by \cite[Theorem 4]{CCUW}, there is a natural \emph{modular logarithmic tropicalization morphism} 
\begin{equation*}
\trop_{g,n}^{log}\colon \calM_{g,n}^{log}\longrightarrow\calM_{g,n}^{trop}
\end{equation*}
 that is strict, smooth, and surjective. It is given by associating to a logarithmic curve $X\rightarrow S$ the family $(\Gamma_{X_s})$ of \emph{dual tropical curves} of each fiber $X_s$ over a geometric point $s$ of $S$. The dual tropical curve $\Gamma_X$ of a logarithmic curve $X$ over a logarithmic point $S$ is defined as follows:
\begin{itemize}
\item the underlying graph $G_{X}$ is the \emph{dual graph} of the stable curve $\underline{X}$, so that its vertices $v$ correspond to the irreducible components $X_v$ of $\underline{X}$, an edge connecting two vertices $v,v'$ to a node connecting the two components $X_v$ and $X_{v'}$ and the legs of $G_X$ emanating from $v$ correspond to the marked points on $X_v$;
\item the vertex weight $h(v)$ is the genus of the normalization $\widetilde{X}_v$ of $X_v$; and 
\item the edge length $\vert e\vert\in \Mbar_{S,s}$ of an edge $e$ of $G(X)$ is the logarithmic deformation parameter $\delta_e\in\Mbar_{S,s}$ at the node $p_e$, as explained in Theorem \ref{thm_logcurves} (ii) above. 
\end{itemize}

\begin{definition}
We define $\calT_{g}^{log}$ to be the fiber product
\begin{equation*}
\calT_{g}^{log}=\calM_{g}^{log}\times_{\calM_{g}^{trop}}\calT_{g}^{trop}
\end{equation*}
over the logarithmic tropicalization morphism $\trop^{log}_{g}\colon \calM_{g}^{log}\rightarrow \calM_{g}^{trop}$ and the natural morphism $\calT_{g}^{trop}\rightarrow \calM_{g}^{trop}$ that forgets the Teichm\"uller marking. 
\end{definition}

Using the above description of $\calM_{g}^{trop}$ as a stack over $\mathbf{LSch}$ the stack $\calT_{g}^{log}$ is the category whose fiber over a logarithmic scheme $S$ is the groupoid of logarithmic curves $X\rightarrow S$ in $\calM_{g}^{log}(S)$ together with a topological equivalence class of a Teichm\"uller marking $\big[\phi_s\colon \pi_1(\Gamma_s)\xrightarrow{\sim} F_g\big]$ on every dual tropical curve $\Gamma_s$ (where $s$ are the geometric points of $S$) that are compatible with \'etale specialization.  

\begin{theorem}\label{thm_logstack=algstack}
The logarithmic stack $\calT_{g}^{log}$ is representable by a pair $\big(\underline{\calT}^{log}_{g}, M_{\calS^{log}_{g,n}}\big)$ consisting of:
\begin{enumerate}[(i)]
\item a Deligne-Mumford stack $\underline{\calT}^{log}_{g}$, that is smooth, universally closed, and locally of finite type over $\Z$; and
\item a fine and saturated logarithmic structure $M_{\calT^{log}_{g}}$ that is associated to the complement of the locus of smooth curves in $\underline{\calT}^{log}_{g}$, which has normal crossings over $\Z$. 
\end{enumerate}
\end{theorem}

\begin{proof}
We may define $\underline{\calT}^{log}_{g}$ as the fibered product 
\begin{equation*}
\underline{\calT}_{g}^{log}=\calMbar_{g}\times_{\underline{\calM}_{g}^{trop}}\underline{\calT}_{g}^{trop}\ .
\end{equation*}
Since $\calT_g^{trop}\rightarrow \calM_g^{trop}$ is strict and surjective, it immediately follows that $\underline{\calT}_{g}^{log}$ is smooth, universally closed, and locally of finite type over $\Z$. We endow $\underline{\calT}_{g}^{log}$ with the logarithmic structure $M_{\calT^{log}_{g}}$ that is associated to the pullback of the boundary divisor of $\calMbar_g$, which has normal crossings, since the boundary divisor on $\calMbar_g$ has normal crossings on the map $\calT_g^{trop}\rightarrow \calM_g^{trop}$ is strict and therefore also smooth (in fact, \'etale locally an isomorphism). By \cite[Theorem 4]{CCUW}, the tropicalization morphism $\calM_{g}^{log}\rightarrow\calM_{g}^{trop}$ is strict and therefore we have 
\begin{equation*}
\calT_g^{log}\simeq \calM_g^{log}\times_{\underline{\calM}_g^{trop}}\underline{\calT}_g^{trop} \ . 
\end{equation*}
This, together with \cite[Theorem 4.5]{Kato_logsmoothcurves}, i.e. Theorem \ref{thm_Katologsmoothcurves} above, implies our claim. 
\end{proof}

\begin{remark}
The moduli stack $\calT^{log}_{g}$ is not separated. The reason is that e.g. in a stable degeneration of a smooth curve, we have exactly one equivalence class of Teichm\"uller markings in the generic fiber and very many in the special fiber, whenever its dual graph has non-trivial cycles. 
\end{remark}

\begin{remark}
It follows a posteriori from Theorem \ref{thm_logtrop=modtrop} that the boundary divisor of $\underline{\calS}_{g}$ even has simple normal crossings, i.e. that all strata of the boundary divisor are smooth over $\Z$. 
\end{remark}

The following Theorem \ref{thm_logquot} lifts Theorem \ref{thm_tropquot} to the logarithmic category. 

\begin{theorem}\label{thm_logquot}
The relative coarse moduli space of the stack quotient $\big[\calT_{g}\big/\Out(F_g)\big]$ over $\calM_{g}^{\log}$ is equivalent to $\calMbar_{g}^{log}$. 
\end{theorem}

\begin{proof}
By Theorem \ref{thm_tropquot} the relative coarse moduli space of the stack quotient $\big[\calT_{g}^{trop}\big/\Out(F_g)\big]$ over $\calM_{g}^{trop}$ is equivalent to $\calM_{g}^{trop}$. The claim is an immediate consequence of this and of the definition of $\calT_{g}^{log}$ as a fibered product $\calM_{g}^{log}\times_{\calM_{g}^{trop}}\calT_{g}^{trop}$. 
\end{proof}


\subsection{From the fundamental category of a logarithmic stack to its tropicalization}
Let $\calX$ be a logarithmic algebraic stack. Denote by $\widetilde{\Pi}_1(\calX)$ the category whose objects are the geometric points $x\rightarrow\calX$ and whose morphisms are \'etale specializations $x\leadsto y$ in $\calX$. We say that an \'etale specialization $x\leadsto y$ is \emph{strict}, if the induced map $\Mbar_{\calX,y}\rightarrow \Mbar_{\calX,x}$ is an isomorphism.

\begin{definition}
The \emph{fundamental category} $\Pi_1(\calX)$ of a logarithmic stack $\calX$ is defined to be the localization of $\widetilde{\Pi}_1(\calX)$ along the class of strict specializations. 
\end{definition}

For a geometric point $x\rightarrow \calX$ we write $\sigma_x$ for the rational polyhedral cone $\Hom\big(\Mbar_{\calX,x},\R_{\geq 0}\big)$. For an \'etale specialization $x\leadsto y$ the induced morphism $\sigma_x\rightarrow \sigma_y$ is a face morphism and, whenever $x\leadsto y$ is strict, this map is an isomorphism. So there is a natural functor $\Pi_1(\calX)\rightarrow\mathbf{RPC}^{f}$ given by the association $x\mapsto \sigma_x$. 

\begin{proposition}\label{prop_Pi=combconestack}
If $\calX$ is a logarithmically smooth (over a base scheme $S$ with trivial logarithmic structure), then the functor $x\mapsto \sigma_x$ makes $\Pi_1(\calX)$ into a category fibered in groupoids over $\mathbf{RPC}^{f}$. 
\end{proposition}

\begin{proof}
We may check the two axioms of a category fibered in groupoids \'etale locally on $\calX$ and so we may assume that $\calX$ is an affine toric variety $\Spec \calO_S[P]$ over $S$ (where $P$ is a fine and saturated monoid). In this case $\Pi_1\big(\Spec \calO_S[P]\big)$ is equivalent to the poset of (generic points of) torus orbits of  $\Spec \calO_S[P]$. Then the statement follows from the order-reversing correspondence between torus orbits of $\Spec\calO_S[P]$ and faces of $\sigma_P=\Hom(P,\R_{\geq 0})$. 
\end{proof}

In other words, Proposition \ref{prop_Pi=combconestack} tells us that $\Pi_1(\calX)$ is a combinatorial cone stack. As explained in Proposition  \ref{prop_coarsemodulispace} above, we may associate to $\Pi_1(\calX)$ its coarse moduli space $\calC_\calX$. So the combinatorial cone stack $\calC_\calX^{comb}$ with faithful monodromy and comes with a strict morphism $\Pi_1(\calX)\rightarrow \calC_\calX$ that is initial among all strict morphisms to cone stacks with faithful monodromy. 

\begin{proposition}\label{prop_factorization}
Let $\calX$ be logarithmically smooth (over a base scheme $S$ with trivial logarithmic structure). Then there is a strict morphism $\calX\rightarrow \calC_\calX$ that is initial among strict morphism to cone stacks with faithful monodromy.
\end{proposition}

Here, we again lift $\calC_\calX$ to the category of logarithmic stacks as an Artin fan. In \cite[Propositon 3.1.1]{ACMW} we see that every (reasonable) logarithmic stack $\calX$ admits a strict morphism $\calX\rightarrow\calA_\calX$ to an Artin fan with faithful monodromy that is initial among all strict morphisms to Artin fans with faithful monodromy. When $\calX$ is logarithmically smooth, Proposition \ref{prop_factorization} tells us that the Artin fan associated to $\calC_\calX$ is equivalent to $\calA_\calX$ so that the diagram
\begin{equation*}\begin{tikzcd}
&\calX\arrow[ld]\arrow[rd]&\\
\calC_\calX\arrow[rr,"\simeq"]& & \calA_\calX 
\end{tikzcd}\end{equation*}
commutes. We refer to $\trop_\calX\colon X\rightarrow\calC_\calX$ as the \emph{logarithmic tropicalization morphism} associated to $\calX$ and to $\calC_\calX$ as the \emph{logarithmic tropicalization of $\calX$.}

\begin{proof}[Proof of Proposition \ref{prop_factorization}]
Let us first assume that $\calX$ is represented by a logarithmically smooth scheme $X$ that is \emph{small}, i.e. that $X$ has a unique closed logarithmic stratum. In this case, the lift of $\calC_X$ is given by $\calA_\sigma$ where $\sigma$ is the rational polyhedral cone $\Hom(\Mbar_X,\R_{\geq 0})$ dual to the characteristic monoid of $X$. The strict morphism $X\rightarrow\calA_\sigma$ is the one associated to the identity under the natural correspondence
\begin{equation*}
\Hom_{\mathbf{LSt}}\big(X,\calA_\sigma\big)=\Hom_{\mathbf{Mon}}\big(S_\sigma,\Mbar_{X}\big) 
\end{equation*}
from \cite[Proposition 5.17]{Olsson_LOG}. We may now continue our argument as in \cite[Proposition 3.1.1]{ACMW} and show that this morphism is initial among strict morphisms to Artin fans with faithful monodromy. In the general situation (when $\calX$ is not small), both $\calA_\calX$ and $\calC_\calX$ arises as colimits of representable morphisms over Olsson's stack $\LOG_S$ of logarithmic structures over $S$ (as introduced in \cite{Olsson_LOG}) and therefore both constructions agree.
\end{proof}

\subsection{Tropicalization of $\calM_{g}^{log}$ and $\calT_{g}^{log}$}

The following Theorem \ref{thm_logtrop=modtrop} will imply Theorem \ref{thm_skel=trop} from the introduction (see Theorem \ref{thm_skel=tropdetail} below). In the case of $\calM_{g}^{log}$ it also rephrases \cite[Theorem 1.3]{Ulirsch_nonArchArtin}.

\begin{theorem}\label{thm_logtrop=modtrop}
The logarithmic tropicalization of $\calM_{g}^{log}$ is isomorphic to the coarse moduli space $M_{g}^{trop}$ and the logarithmic tropicalization of $\calT_{g}^{log}$ is equivalent to $\calT_{g}^{trop}$ so that the natural diagram
\begin{equation}\begin{tikzcd}\label{eq_logtropdiag}
&\calT_{g}^{log} \arrow[ld,"\trop_{g}"']\arrow[rrd,"\trop_{\calT_{g}^{log}}"]\arrow[dd]&\\
\calT_{g}^{trop} \arrow[rrr,"\simeq", crossing over]\arrow[dd] & & & \calC_{\calT_{g}^{log}}\arrow[dd]\\
& \calM_{g}^{log}\arrow[ld,"\trop_{g}"']\arrow[rrd,"\trop_{\calM_{g}^{log}}"] &\\
\calM_{g}^{trop} \arrow[rrr]& & & \calC_{M_{g}^{log}}
\end{tikzcd}\end{equation}
commutes.
\end{theorem}

\begin{proof}
The first step of the proof consists of a stack-theoretic and logarithmic reinterpretation of the proof of the main result in \cite{ACP}. It is based on the following facts, which be found e.g. in \cite{ACGII} and \cite{ACP}:
\begin{itemize}
\item The boundary strata of $\calMbar_{g}$ are in natural one-to-one correspondence with stable weighted graphs of genus $g$.
\item An \'etale specialization $\eta_{G}\leadsto \eta_{G'}$ of generic points of boundary strata that is not an isomorphism corresponds to a weighted edge contraction $G\rightarrow G'$.
\item There is an isomorphism $\Mbar_{\eta_G}\simeq\N^{E(G)}$ such that the group of automorphisms of $\Mbar_{\eta_G}$ induced by self-specializations $\eta_G\xrightarrow{\sim}\eta_G$ (called the \emph{monodromy group} of the stratum in \cite{ACP}) agrees with the group of permutations of $\N^{E(G)}$ induced by automorphisms of $G$. 
\end{itemize}
These three facts together imply that there is a natural equivalence between the coarse moduli space of $\calM_{g}^{trop}$ and $\calC_{\calM_g^{log}}$. Since, by \cite[Theorem 4]{CCUW} the tropicalization map $\trop_g$ is strict, the lower triangular diagram in \eqref{eq_logtropdiag} commutes by the universal property from Proposition \ref{prop_factorization}. 

In the second step we notice that by construction of $\calT_g^{log}$ as a fiber product $\calM_{g}^{log}\times_{\calM_{g}^{trop}}\calT_{g}^{trop}$, we also have the following facts:
\begin{itemize}
\item The logarithmic strata of $\calT_g^{log}$ are in a natural one-to-one correspondence with pairs consisting of a stable weighted graph of genus $g$ and a topological equivalence class of Teichm\"uller markings;
\item An \'etale specialization $\eta_{G,[\phi]}\leadsto \eta_{G',[\phi']}$ of generic points of boundary strata that is not an isomorphism corresponds to a weighted edge contraction $G\rightarrow G'$ that makes the resulting diagram of Teichm\"uller markings commute up to inner automorphisms.
\item Both the monodromy group and the automorphism group of $(G,[\phi])$ are trivial.
\end{itemize}
Thus there is a natural equivalence between $\calT_{g}^{trop}$ and the $\calC_{\calT_g^{log}}$. Since the tropicalization map $\trop_g$ is strict, as a base change of a strict tropicalization map, the upper triangular diagram is commutative by the universal property in Proposition \ref{prop_factorization}. 

Finally, we observe that the two back squares in \eqref{eq_logtropdiag} commute by construction and thus the front square commutes, since $\trop_g$ is surjective.
\end{proof}



\section{Skeletons and tropicalization}

Throughout this section let $K$ be an algebraically closed non-Archimedean field with valuation ring $R$. In this section we construct $\calSbar_g$ by applying Raynaud's generic fiber functor to $\calS_g^{log}$ and identify its non-Archimedean tropicalization map with a strong deformation retraction onto the skeleton.


\subsection{Extended (generalized) cone complexes}
Given a rational polyhedral cone $\sigma$, its \emph{canonical extension} is defined to be 
\begin{equation*}
\sigmabar=\Hom(S_\sigma,\R_{\geq 0}) \ .
\end{equation*}
One may think of $\sigmabar$ as a compactification of $\sigma$ given by adding further faces at infinity (see e.g. \cite[Section 5]{Rabinoff_Newtonpolygons} for details).  As in \cite{ACP,Ulirsch_functroplogsch}, we define the \emph{canonical extension} $\Sigmabar$ of a (generalized) cone complex $\Sigma$ as the colimit of the diagram that arises when we replace all cones in the defining diagram by their canonical extensions. 

In the following we write $M_{g}^{trop}$ and $T_g^{trop}$ for the coarse moduli spaces of $\calM_g^{trop}$ and $\calT_g^{trop}$ in the sense of Section \ref{section_coarsemodulispaces} respectively. Both $M_g^{trop}$ and $T_g^{trop}$ are objects in the $2$-category of cone stacks with faithful monodromy and in the category of generalized cone complexes. Denote by $\Mbar_g^{trop}$ and $\Tbar_g^{trop}$ their canonical extensions. The points of $M_g^{trop}$ are in natural one-to-one correspondence with stable tropical curves of genus $g$ with real edge lengths and the points of $T_g^{trop}$ are pairs consisting of a stable tropical curve of genus $g$ a topological equivalence class of Teichm\"uller markings. Their canonical extensions parametrize the same data, only we allow the edge lengths of the tropical curves to take non-zero values in the additive monoid $\Rbar_{\geq 0}=\R_{\geq 0}\sqcup\{\infty\}$ (see \cite[Section 4]{ACP} for details). \emph{Outer space} $\CV_g^{trop}$ in the sense of Culler-Vogtmann \cite{CullerVogtmann} as the locus of metric graphs in $T_g^{trop}$, i.e. as the locus of tropical curves with all vertex weights equal to zero. 

\subsection{Non-Archimedean tropicalization of $\calMbar_{g}$}
Denote by $\calMbar_{g}^{an}$ the non-Archimedean analytic stack associated to $\calMbar_{g,K}$. We refer the reader to \cite{Yu_Gromovcompactness, Ulirsch_tropisquot} for the basic definitions of non-Archimedean analytic stacks and implicitly identify $\calMbar_{g}^{an}$ with its underlying topological space, as introduced in \cite[Section 3]{Ulirsch_tropisquot}. 

There is a natural \emph{non-Archimedean tropicalization map}
\begin{equation*}
\trop_g\colon\calMbar_g^{an}\longrightarrow \Mbar_g^{an}
\end{equation*}
that associates to a point in $\calMbar_g^{an}$, corresponding to a stable curve $X$ over a non-Archimedean extension $L$ of $K$, its \emph{dual tropical curve $\Gamma_X$}. Let us explain this:

The valuative criterion for properness, applied to $\calMbar_g$ tells us that there is a finite extension $L'$ of $L$ such that the base change $X_{L'}$ admits a stable model $\calX$ over the valuation ring $R'$ of $L'$. In other words, there is a proper and flat scheme $\calX$ over $R'$ with reduced fibers of dimension one such that the generic fiber is isomorphic to $X_{L'}$ and the special fiber $\calX_0$ is a stable nodal curve over the residue field of $R'$. The \emph{dual graph} $G_{\calX_0}$ of $\calX_0$ is the graph with vertices are the irreducible components of $\calX_0$ and an edge between two vertices for every node connecting the corresponding components. A vertex weight $h\colon V(G_{\calX_0})\rightarrow\Z_{\geq 0}$ on $G_{\calX_0}$ associates to a vertex the genus of the normalization of the corresponding component. \'Etale locally around every node, the scheme $\calX$ is given by $xy=r_e$ for two coordinates $x$ and $y$ and an element $r\in R'$. The edge length on $G_{\calX_0}$ is given by $\vert e\vert=\val(r_e)$. Notice hereby that for the edges $e$ corresponding to nodes that were already present in the generic fiber, we always have $\vert e\vert =0$. The \emph{dual tropical curve} $\Gamma_X$ is the tropical curve given by the tuple $(G_{\calX_0}, h, \vert.\vert)$ (with edge lengths in $\Rbar_{\geq 0}$). 

It follows a posteriori from the identification of this map with the strong deformation retraction onto the non-Archimedean skeleton in Theorem \ref{thm_skel=tropdetail} below that $\trop_g$ is well-defined and continuous. 

\subsection{Raynaud's generic fiber functor} Berkovich analytification defines a functor from the category of schemes locally of finite type over $K$ to the category non-Archimedean analytic spaces. By \cite{Berkovich_book} a scheme $X$ locally of finite type over $K$ is separated if and only $X^{an}$ is a Hausdorff space. Since we are considering the non-separated stack $\calS_g^{log}$, we therefore want to work with a different analytification functor, known as \emph{Raynaud's generic fiber functor} (as introduced in \cite{Berkovich_vanishingcyclesII}). 

It associates to a flat scheme $\mathscr{X}$ locally of finite type over $R$ a Berkovich analytic space $\mathscr{X}_\eta$ that functions as an analytic generic fiber of the formal completion of $\mathscr{X}$ along the maximal ideal of $R$. Suppose that $\mathscr{X}=\Spec \mathscr{A}$ is affine and write $X=\Spec A$ for its generic fiber, where $A=\mathscr{A}\otimes_R K$. In this case, the Raynaud generic fiber is the affinoid domain in $X^{an}$ whose points are those seminorms $\vert .\vert_x$ on $A$ for which $\vert a\vert_x\leq 1$ for all $a\in \mathscr{A}$, i.e. those seminorms that extend to a bounded seminorm on $\mathscr{A}$. 

In general, when $\mathscr{X}$ is not affine, the Raynaud generic fiber $\mathscr{X}_\eta$ is given by glueing the affine patches of $\mathscr{X}$. A point in $\mathscr{X}_\eta$ may be represented by a morphism $\Spec R'\rightarrow \mathscr{X}$ from a valuation ring $R'$ that extends $R$. Two morphisms $\Spec R'\rightarrow\mathscr{X}$ and $\Spec R''\rightarrow \mathscr{X}$ represent the same point if and only if there is a valuation ring $\Omega$ that extends both $R'$ and $R''$ such that the induced diagram
\begin{equation*}\begin{tikzcd}
\Spec \Omega\arrow[r]\arrow[d]&  \Spec R'\arrow[d]\\
\Spec R''\arrow[r]& \mathscr{X}
\end{tikzcd}\end{equation*}
commutes. By the valuative criteria, the Raynaud generic fiber $\mathscr{X}_\eta$ is an analytic domain in the Berkovich analytic space $\mathscr{X}_K^{an}$ associated to the algebraic geometric generic fiber $\mathscr{X}_K$ of $\mathscr{X}$ if and only if $\mathscr{X}$ is separated over $R$, and $\mathscr{X}_\eta$ is isomorphic to $\mathscr{X}_K^{an}$ if and only if $\mathscr{X}$ is proper over $R$.

In \cite{Yu_Gromovcompactness}, Yu has extend the generic fiber functor $(.)_\eta$ to a functor
\begin{equation*}
(.)_\eta\colon \mathbf{Alg.Stacks}_{\textrm{flat, loc.f.t.}/k}\longrightarrow \mathbf{An.Stacks}_k
\end{equation*}
such that, whenever $\mathscr{X}=\big[\mathscr{U}\big/\mathscr{R}\big]$ is a groupoid presentation of an algebraic stack $\mathscr{X}$ that is flat and locally of finite type over $R$, we have 
\begin{equation*}
\mathscr{X}_\eta \simeq\big[\mathscr{U}_\eta/\mathscr{R}_\eta\big] \ . 
\end{equation*}
Again a point in $\mathscr{X}_\eta$ may be represented by a morphism $\Spec R'\rightarrow \mathscr{X}$ from a valuation ring $R'$ that extends $R$. Two morphisms $\Spec R'\rightarrow\mathscr{X}$ and $\Spec R''\rightarrow \mathscr{X}$ then represent the same point if and only if there is a valuation ring $\Omega$ that extends both $R'$ and $R''$ such that the induced diagram
\begin{equation*}\begin{tikzcd}
\Spec \Omega\arrow[r]\arrow[d]&  \Spec R'\arrow[d]\\
\Spec R''\arrow[r]& \mathscr{X}
\end{tikzcd}\end{equation*}
is $2$-commutative. Yet again, by the valuative criteria, the Raynaud generic fiber $\mathscr{X}_\eta$ is an analytic domain in the Berkovich analytic stack $\mathscr{X}_K^{an}$ if and only if $\mathscr{X}$ is separated over $R$, and $\mathscr{X}_\eta$ is equivalent to $\mathscr{X}_K^{an}$ if and only if $\mathscr{X}$ is proper over $R$.

\subsection{Non-Archimedean Schottky space and its tropicalization}

Recall from Section \ref{section_logTeichmueller} that $\underline{\calT}_{g}^{log}$ is the underlying algebraic stack of $\calT_g^{log}=\calM_g^{log}\times_{\calM_g^{trop}}\calT_g^{trop}$. It is a smooth and universally closed, but not separated Deligne-Mumford stack over $\Spec\Z$. We write $\underline{\calT}_{g,R}^{log}$ for the base change of $\underline{\calT}_{g}^{log}$ to the valuation ring $R$. 

\begin{definition}
The \emph{extended non-Archimedean Teichm\"uller space} $\calTbar_{g}$ over $K$ is defined to be
\begin{equation*}
\calTbar_{g}=\underline{\calT}_{g,R,\eta}^{log} \ .
\end{equation*}
\end{definition}

A point in $\calTbar_{g}$ is represented by a tuple $(X,\phi)$ where $X$ is a stable curve of genus $g$ over a non-Archimedean extension $L$ of $K$ together with a Teichm\"uller marking of the dual tropical curve $\Gamma_X$ of its stable reduction. It is well-known that the geometric realization $\vert \Gamma_X\vert$ of $\Gamma_X$ as a metric space arises as the \emph{minimal skeleton} of the non-Archimedean analytic space $X_L^{an}$ (see \cite{Berkovich_book, BPR} for details).
A Teichm\"uller marking on $\Gamma_X$ amounts to the choice of an equivalence $\pi_1(X^{an})=\pi_1\big(\vert\Gamma_X\vert\big)\xrightarrow{\sim} F_{b_1}$ where $b_1$ is the Betti number of both $X^{an}$ and its minimal skeleton $\vert \Gamma_X\vert$. So, in particular, we have a natural \emph{non-Archimedean tropicalization map} 
\begin{equation*}\begin{split}
\trop_{g}\colon \calTbar_{g}&\longrightarrow \Tbar_{g}^{an}\\
\Big(X,\phi\colon \pi_1\big(X^{an}\big)\xrightarrow{\sim} F_{b_1}\Big)&\longmapsto \Big(\Gamma_X,\phi\colon \pi_1\big(\vert\Gamma_X\vert\big)\xrightarrow{\sim}F_{b_1}\Big) \ .
\end{split}\end{equation*}

We define $\calT_{g}$ as the locus of pairs $(X,\phi)$ where $X$ is a smooth curve and refer to it as the \emph{non-Archimedean Teichm\"uller space} over $K$. It also arises as the preimage $\trop_{g}^{-1}(T_g^{trop})$ of the locus of non-extended tropical curves (with edge lengths only in $\R_{>0}$ instead of $\Rbar_{>0}=\R_{>0}\sqcup\{\infty\}$) and as the preimage of the analytification $\calM_{g}^{an}$ of $\calM_{g}$ under the natural map $\calTbar_{g}\rightarrow\calMbar_{g}^{an}$ that forgets the marking. 

\begin{proof}[Proof of Theorem \ref{thm_calTbar_g}]
Since $\underline{\calT}_g^{log}$ is smooth over $\Z$, its base change to $R$ is smooth over $R$ and therefore the Raynaud generic fiber is $G$-smooth over $K$. Let $G$ be stable weighted graph of genus $g$ and consider the rational polyhedral cone $\sigma_G=\R_{\geq 0}^{E(G)}$, which parametrizes all tropical curves whose underlying graph is a weighted edge contraction of $G$. Write $U_{\sigma_G}$ for the affine toric variety $\Spec \Z[S_{\sigma_G}]$ associated to $\sigma_G$. There is a natural morphism $\sigma_G\rightarrow\calM_g^{trop}$ that induces a morphism $U_\sigma\rightarrow \calM_g^{trop}$ (here the right hand side is the really the Artin fan $\calA_{\calM_g^{trop}}$). The base change 
\begin{equation*}
\widetilde{U}_{\sigma_G}=U_{\sigma_G}\times_{\calM_g^{trop}}\calT_g^{trop}
\end{equation*}
is a non-separated toric variety. Each of its maximal torus-invariant open affine subsets is isomorphic to $U_{\sigma_G}$, since $\calT_g^{trop}\rightarrow{\calM_g^{trop}}$ is a strict cover. Thus, applying the Raynaud generic fiber functor, we obtain a morphism $\widetilde{U}_{\sigma_G,R,\eta}\rightarrow U_{\sigma_G,R,\eta}$ without boundary. Since $\calTbar_g\rightarrow\calMbar_g^{an}$ arises \'etale locally as a base change of such morphisms, it also without boundary. This proves that $\calTbar_g$ is without boundary over $K$, since $\calMbar_g^{an}$ is. Since $\calTbar_g$ is $G$-smooth and without boundary, it is also smooth. 

In order to show that $\calTbar_g$ is separated we show that the diagonal morphism $\Delta\colon \calTbar_g\rightarrow\calTbar_g\times\calTbar_g$ is proper. Let $U\rightarrow \calTbar_g\times\calTbar_g$ be a morphism from a strict analytic space $U$. This corresponds to a flat and proper analytic family $X\rightarrow U$ of stable curves together with a family of Teichm\"uller markings of the dual tropical curves, compatible with \'etale specialization. The fiber product $\calTbar_g\times_{\calTbar_g\times\calTbar_g}U$ is representable by an analytic space $V$ whose points are exactly the triples $(x,y,\phi)$ consisting of two points $x,y\in U$ and an isomorphism $\phi\colon X_x\xrightarrow{\sim}X_y$ that is compatible with the Teichm\"uller markings on the dual tropical curves. Since the diagonal morphism $\calMbar_g^{an}\rightarrow\calMbar_g^{an}\times\calMbar_g^{an}$ is proper, the base change of $V\rightarrow U$ to any affinoid domain in $U$ is finite and this implies that $V\rightarrow U$ is without boundary. Any compact subset $A$ in $U$ will be a subset of a finite union of affinoid domains and thus this also implies that the preimage of $A$ in $V$ is compact. Therefore $V\rightarrow U$ is a proper morphism and, since $U$ was chosen arbitrarily, this implies that the diagonal morphism $\Delta\colon \calTbar_g\rightarrow\calTbar_g\times\calTbar_g$ is proper. 
\end{proof}

\begin{proof}[Proof of Theorem \ref{thm_quotient}]
This immediately follows from Theorem \ref{thm_logquot} and the fact that both the functors $(.)\otimes \Spec R$ and $(.)_\eta$ preserve ($2$-)colimits.
\end{proof}


\subsection{Non-Archimedean skeletons of stacky normal crossing pairs with good reduction}\label{section_stackyskeleton}
Let $(\scrX,\scrD)$ be a \emph{strictly semistable pair} over $R$, consisting a flat scheme $\scrX$ locally of finite type over $\Spec R$ whose generic fiber is smooth and whose special fiber has strict normal crossings and an effective strict normal crossing divisor $\scrD=\calD_1+\cdots +\calD_s$ on $\scrX$ which includes the special fiber. In \cite[Section 4]{GRW} the authors expands on a construction of Berkovich in \cite{Berkovich_smooth=locallycontractible} and show that the generic fiber $\scrX_\eta$ admits a strong deformation retraction $\rho_{(\scrX,\scrD)}\colon \scrX_\eta\rightarrow\scrX_\eta$ onto a closed subset $\Sigmabar(\scrX,\scrD)$ of $\scrX_\eta$, the \emph{skeleton of the strictly semistable pair $(\scrX,\scrD)$}.

\begin{remark}
In \cite{GRW} the authors always require that $\scrX$ be proper over $R$. This assumption is not necessary, if we work with the Raynaud generic fiber $\scrX_\eta$ instead of the Berkovich analytic space of the algebraic generic fiber of $\scrX$. 
\end{remark}

For the purpose of this article we assume that $\scrX$ is smooth over $R$, i.e. that $\scrX$ has \emph{good reduction}. In this case, the skeleton $\Sigmabar(\scrX,\scrD)$ canonically has the structure of an extended cone complex: The space $\calX$ is naturally stratified by locally closed subschemes; the strata are the connected components of of the smooth locus of $\calD_{i_1}\cap\cdots\cap\calD_{i_k}$. If $\scrU\subseteq\scrX$ is \emph{small} open subset of $\scrX$ (or \emph{a building block} in the terminology of \cite[Section 4.3]{GRW}), i.e. if it contains a unique closed stratum, the skeleton is naturally homeomorphic to the extended cone $\sigmabar_\scrX=\Hom\big(\Mbar_{\scrD},\Rbar_{\geq 0}\big)$ and the retraction $\rho_{(\scrX,\scrD)}$ is given by 
\begin{equation}\label{eq_retraction}
x\longmapsto \big(\Mbar_{\scrD}\rightarrow\calO_{\mathscr{U}}/\calO_{\mathscr{U}}^{\ast}\xrightarrow{-\log\vert.\vert_x}\Rbar_{\geq 0}\big)\ .
\end{equation}
Here $M_\scrD$ denotes the natural divisorial logarithmic structure on $\scrX$ associated to the divisor $\scrD$, i.e. the sheaf of sections of $s\in\calO_X$ such that $s\vert_{\scrU-\scrD}\in\calO_\calX^\ast$; this way $\Mbar_\scrD=M_\scrD/\calO_\scrX^\ast$ is identified with the sheaf of effective Cartier divisors with support in $\scrD$. In general, for every strictly semistable pair $(\scrX,\scrD)$ there is a cover by small open subsets $\scrU_i$ as well as a cover $\scrV_k^{ij}$ of $\scrU_i\cap \scrU_j$ by small open subsets $\scrV_k^{ij}$. Then the  extended cone complex $\Sigmabar(\scrX,\scrD)$ is the colimit of the induced diagram of proper face morphisms $\sigmabar_{V_k^{ij}}\hookrightarrow \sigmabar_{U_i}$ and the retraction map is defined by \eqref{eq_retraction} on every small open subset $\scrU\subseteq \scrX$. 

In the following Lemma \ref{lemma_skeletonncpairs}, we generalize this construction to the case of a \emph{stacky normal crossing pair $(\scrX,\scrD)$ with good reduction}, to the case of a Deligne-Mumford stack $\scrX$ that is smooth over $\Spec R$ together with   an effective Cartier divisor $\scrD$ on $\scrX$ that has (stack-theoertically) normal crossings. We say that an \'etale morphism $f\colon \scrU\rightarrow \scrX$ is \emph{small}, if $(\scrU,f^\ast\scrD)$ is a strictly semistable pair with good reduction that is small. 

\begin{lemma}\label{lemma_skeletonncpairs}
Let $\mathscr{X}$ be a Deligne-Mumford stack that is smooth over $R$ and let $\mathscr{D}$ be an effective Cartier divisor on $\mathscr{X}$ with (stack-theoretically) normal crossings. Then there is a strong deformation retraction $\rho_{(\mathscr{X},\mathscr{D})}\colon \mathscr{X}_\eta\rightarrow \mathscr{X}_\eta$ onto a closed subset $\Sigmabar(\mathscr{X},\mathscr{D})$, the non-Archimedean skeleton associated to $(\mathscr{X},\mathscr{D})$. The skeleton $\Sigmabar(\mathscr{X},\mathscr{D})$ is naturally homeomorphic to the extended generalized cone complex associated to $\calC(\mathscr{X},\mathscr{D})$ and, on a small \'etale neighborhood $\mathscr{U}\rightarrow \mathscr{X}$, the retraction map $\rho_{(\mathscr{X},\mathscr{D})}$ is given by the non-Archimedean tropicalization map 
\begin{equation*}\begin{split}
\mathscr{U}_\eta&\longrightarrow \sigmabar_{\mathscr{U}}=\Hom(\Mbar_{\mathscr{U}},\Rbar_{\geq 0})\\
x&\longmapsto \big(\Mbar_{\mathscr{U}}\rightarrow\calO_{\mathscr{U}}/\calO_{\mathscr{U}}^{\ast}\xrightarrow{-\log\vert.\vert_x}\Rbar_{\geq 0}\big)\ . 
\end{split}\end{equation*}
\end{lemma}

\begin{proof}
Consider the category $Q(\scrX,\scrD)$ whose objects are small \'etale morphisms $f\colon \scrU\rightarrow \scrX$ and whose morphisms are commutative diagrams 
\begin{equation}\label{eq_smalldiagram}\begin{tikzcd}
\scrU \arrow[rd,"f"']\arrow[rr,"h"]&& \scrV\arrow[ld,"g"]\\
& \scrX&
\end{tikzcd}\end{equation}
such that $f^\ast\scrD=(g\circ h)^\ast\scrD$. There is a natural functor $Q(\scrX,\scrD)\rightarrow \mathbf{RPC}^{f}$ that is given by $\scrU\mapsto \sigma_U$. Denote by $M_\scrD$ the divisorial logarithmic structure on $\scrX$ associated to $\scrD$ and note that $(\scrX,M_\scrD)$ is logarithmically smooth over $\Spec R$ (with the trivial logarithmic structure). The generalized cone complex associated to $\calC(\scrX,M_\scrD)$ is precisely the colimit over the diagram of all $\sigma_\scrU$ in $Q(\scrX,\scrD)$, since every geometric point has of $\scrX$ has a small open \'etale neighborhood, in which it is in the deepest stratum.

The underlying topological space of $\scrX_\eta$ is the colimit of all $f_\eta\colon\scrU_\eta\rightarrow\scrX_\eta$ for $f\colon \scrU\rightarrow\scrX$ in $Q(\scrX,\scrD)$. Moverover, for every diagram \eqref{eq_smalldiagram} the induced diagram of retraction maps
\begin{equation*}\begin{tikzcd}
\scrU_\eta \arrow[rr,"\rho_\scrU"] \arrow[d,"h_\eta"']& & \scrU_\eta \arrow[d,"h_\eta"]\\
\scrV_\eta \arrow[rr,"\rho_\scrV"] & & \scrV_\eta
\end{tikzcd}\end{equation*}
is commutative. So the $\rho_\scrU$ descend to a retraction map $\scrX_\eta\rightarrow \scrX_\eta$, whose image is defined to be the skeleton $\Sigmabar(\scrX,\scrD)$ of $\scrX$. It is precisely the colimit of all $\sigmabar_\scrU$, taken over the small \'etale neighborhoods $f\colon \scrU\rightarrow \scrX$ in $Q(\scrX,\scrD)$ and therefore isomorphic to the canonical extension of the generalzed cone complex associated to $\calC(\scrX,M_{\scrD})$. 

In order to show that there is a strong homotopy equivalence between $\rho_\scrX$ and the identity on $\scrX_\eta$, we need to observe that the induced diagram of homotopies on \eqref{eq_smalldiagram} is commutative. This argument has already been carried out in \cite[Section 3.3]{Thuillier_toroidal} and \cite[Proposition 6.1.4]{ACP} over base fields with trivial absolute; it carries over to our situation without any changes.
\end{proof}

\subsection{Skeletons of $\calTbar_g$ and $\calMbar_g^{an}$}

The following Theorem \ref{thm_skel=tropdetail} expands on Theorem \ref{thm_skel=trop} from the introduction and generalizes the main result of \cite{ACP} to the case of a non-trivially valued base field.

\begin{theorem}\label{thm_skel=tropdetail}
The skeleton of $\calMbar_{g}^{an}$ is isomorphic to $\Mbar_{g}^{trop}$ and the skeleton of $\calTbar_{g}$ is isomorphic to $\calTbar_{g}^{trop}$ so that the natural diagram 
\begin{equation}\begin{tikzcd}\label{eq_tropdiagnonArch}
&\calTbar_{g} \arrow[ld,"\trop_{g}"']\arrow[rrd,"\rho_{\calTbar_{g}}"]\arrow[dd]&\\
\calTbar_{g}^{trop} \arrow[rrr,"\simeq", crossing over]\arrow[dd] & & & \Sigmabar({\calTbar_{g}})\arrow[dd]\\
& \calMbar_{g}^{an}\arrow[ld,"\trop_{g}"']\arrow[rrd,"\rho_{\calMbar_{g}}"] &\\
\Mbar_{g}^{trop} \arrow[rrr,"\simeq"]& & & \Sigmabar({\calMbar_{g}})
\end{tikzcd}\end{equation}
commutes.
\end{theorem}

\begin{proof}[Proof of Theorem \ref{thm_skel=tropdetail}]
Use Theorem \ref{thm_logtrop=modtrop} together with Lemma \ref{lemma_skeletonncpairs} applied to $\calMbar_{g}^{an}$ and $\calTbar_{g}$. The commutativity of  diagram \eqref{eq_tropdiagnonArch} follows form the commutativity of diagram \eqref{eq_logtropdiag} in Theorem \ref{thm_logtrop=modtrop}.
\end{proof}

\begin{proof}[Proof of Corollary \ref{cor_skel=trop}]
For a stacky normal crossing pair $(\calX,\calD)$, the homotopy between $\rho_{(\calX,\calD)}$ and identity preserves the fibers of $\rho_{(\calX,\calD)}$; this may be checked \'etale locally only toric varieties, on which the homotopy is defined via the torus operation. Using Theorem \ref{thm_skel=tropdetail}, we may therefore restrict the homotopy to the locus of smooth Mumford curves in $\calTbar_g$ and find that the tropicalization defines a strong deformation retraction onto Culler-Vogtmann Outer space $\CV_g$.
\end{proof}

\begin{remark}
Denote by $\CVbar_g$ the canonical extension of Culler-Vogtmann Outer space, where we allow the edge lengths of graphs to attain value $\infty$ (i.e. the closure of $\CV_g$ in $\Tbar_{g}^{trop}$). The same argument as in the proof of Corollary \ref{cor_skel=trop} shows that $\CVbar_g$ is a strong deformation retract of the locus $\calTbar^{\textrm{Mum}}_g$ of stable Mumford curves in $\calTbar_g$.
\end{remark}









\bibliographystyle{amsalpha}
\bibliography{biblio}{}

\newcommand{\etalchar}[1]{$^{#1}$}
\providecommand{\bysame}{\leavevmode\hbox to3em{\hrulefill}\thinspace}
\providecommand{\MR}{\relax\ifhmode\unskip\space\fi MR }
\providecommand{\MRhref}[2]{%
  \href{http://www.ams.org/mathscinet-getitem?mr=#1}{#2}
}
\providecommand{\href}[2]{#2}
\begin{thebibliography}{RSPW19b}

\bibitem[Abi77]{Abikoff_degRS}
William Abikoff, \emph{Degenerating families of {R}iemann surfaces}, Ann. of
  Math. (2) \textbf{105} (1977), no.~1, 29--44.

\bibitem[ACG11]{ACGII}
Enrico Arbarello, Maurizio Cornalba, and Phillip~A. Griffiths, \emph{Geometry
  of algebraic curves. {V}olume {II}}, Grundlehren der Mathematischen
  Wissenschaften [Fundamental Principles of Mathematical Sciences], vol. 268,
  Springer, Heidelberg, 2011, With a contribution by Joseph Daniel Harris.

\bibitem[ACG{\etalchar{+}}13]{Abramovichetal_logmoduli}
Dan Abramovich, Qile Chen, Danny Gillam, Yuhao Huang, Martin Olsson, Matthew
  Satriano, and Shenghao Sun, \emph{Logarithmic geometry and moduli}, Handbook
  of moduli. {V}ol. {I}, Adv. Lect. Math. (ALM), vol.~24, Int. Press,
  Somerville, MA, 2013, pp.~1--61.

\bibitem[ACGS17]{ACGS}
Dan Abramovich, Qile Chen, Mark Gross, and Bernd Siebert, \emph{Decomposition
  of degenerate {Gromov}-{Witten} invariants}, arXiv:1709.09864 [math] (2017).

\bibitem[ACM{\etalchar{+}}16]{ACMUW}
Dan Abramovich, Qile Chen, Steffen Marcus, Martin Ulirsch, and Jonathan Wise,
  \emph{Skeletons and fans of logarithmic structures}, Nonarchimedean and
  tropical geometry, Simons Symp., Springer, [Cham], 2016, pp.~287--336.

\bibitem[ACMW17]{ACMW}
Dan Abramovich, Qile Chen, Steffen Marcus, and Jonathan Wise, \emph{Boundedness
  of the space of stable logarithmic maps}, J. Eur. Math. Soc. (JEMS)
  \textbf{19} (2017), no.~9, 2783--2809.

\bibitem[ACP15]{ACP}
Dan Abramovich, Lucia Caporaso, and Sam Payne, \emph{The tropicalization of the
  moduli space of curves}, Ann. Sci. \'{E}c. Norm. Sup\'{e}r. (4) \textbf{48}
  (2015), no.~4, 765--809.

\bibitem[And03]{Andre_periods}
Yves Andr\'{e}, \emph{Period mappings and differential equations. {F}rom {$\Bbb
  C$} to {$\Bbb C_p$}}, MSJ Memoirs, vol.~12, Mathematical Society of Japan,
  Tokyo, 2003, T\^{o}hoku-Hokkaid\^{o} lectures in arithmetic geometry, With
  appendices by F. Kato and N. Tsuzuki.

\bibitem[AOV11]{AOV}
Dan Abramovich, Martin Olsson, and Angelo Vistoli, \emph{Twisted stable maps to
  tame {A}rtin stacks}, J. Algebraic Geom. \textbf{20} (2011), no.~3, 399--477.

\bibitem[AW18]{AW}
Dan Abramovich and Jonathan Wise, \emph{Birational invariance in logarithmic
  {G}romov-{W}itten theory}, Compos. Math. \textbf{154} (2018), no.~3,
  595--620.

\bibitem[Bas93]{Bass_graphsofgroups}
Hyman Bass, \emph{Covering theory for graphs of groups}, J. Pure Appl. Algebra
  \textbf{89} (1993), no.~1-2, 3--47.

\bibitem[BCG{\etalchar{+}}19]{BCGGM3}
Matt Bainbridge, Dawei Chen, Quentin Gendron, Samuel Grushevsky, and Martin
  Möller, \emph{The moduli space of multi-scale differentials},
  arXiv:1910.13492 [math] (2019).

\bibitem[Ber90]{Berkovich_book}
Vladimir~G. Berkovich, \emph{Spectral theory and analytic geometry over
  non-{A}rchimedean fields}, Mathematical Surveys and Monographs, vol.~33,
  American Mathematical Society, Providence, RI, 1990.

\bibitem[Ber96]{Berkovich_vanishingcyclesII}
\bysame, \emph{Vanishing cycles for formal schemes. {II}}, Invent. Math.
  \textbf{125} (1996), no.~2, 367--390.

\bibitem[Ber99]{Berkovich_smooth=locallycontractible}
\bysame, \emph{Smooth {$p$}-adic analytic spaces are locally contractible},
  Invent. Math. \textbf{137} (1999), no.~1, 1--84.

\bibitem[BPR16]{BPR}
Matthew Baker, Sam Payne, and Joseph Rabinoff, \emph{Nonarchimedean geometry,
  tropicalization, and metrics on curves}, Algebr. Geom. \textbf{3} (2016),
  no.~1, 63--105.

\bibitem[Cap13]{Caporaso_comparingmoduli}
Lucia Caporaso, \emph{Algebraic and tropical curves: comparing their moduli
  spaces}, Handbook of moduli. {V}ol. {I}, Adv. Lect. Math. (ALM), vol.~24,
  Int. Press, Somerville, MA, 2013, pp.~119--160.

\bibitem[CCUW17]{CCUW}
Renzo Cavalieri, Melody Chan, Martin Ulirsch, and Jonathan Wise, \emph{A moduli
  stack of tropical curves}, arXiv:1704.03806 [math] (2017), Forum Math. Sigma,
  to appear.

\bibitem[CMV13]{CMV}
Melody Chan, Margarida Melo, and Filippo Viviani, \emph{Tropical
  {T}eichm\"{u}ller and {S}iegel spaces}, Algebraic and combinatorial aspects
  of tropical geometry, Contemp. Math., vol. 589, Amer. Math. Soc., Providence,
  RI, 2013, pp.~45--85.

\bibitem[CT09]{ConradTemkin_algspaces}
Brian Conrad and Michael Temkin, \emph{Non-{A}rchimedean analytification of
  algebraic spaces}, J. Algebraic Geom. \textbf{18} (2009), no.~4, 731--788.

\bibitem[CV86]{CullerVogtmann}
Marc Culler and Karen Vogtmann, \emph{Moduli of graphs and automorphisms of
  free groups}, Invent. Math. \textbf{84} (1986), no.~1, 91--119.

\bibitem[FvdP04]{FresnelvanderPut}
Jean Fresnel and Marius van~der Put, \emph{Rigid analytic geometry and its
  applications}, Progress in Mathematics, vol. 218, Birkh\"{a}user Boston,
  Inc., Boston, MA, 2004.

\bibitem[Ger83a]{Gerritzen_Siegel}
Lothar Gerritzen, \emph{{$p$}-adic {S}iegel halfspace}, Study group on
  ultrametric analysis, 9th year: 1981/82, {N}o. 3 ({M}arseille, 1982), Inst.
  Henri Poincar\'{e}, Paris, 1983, pp.~Exp. No. J9, 7.

\bibitem[Ger83b]{Gerritzen_TeichmuellerSiegel}
\bysame, \emph{{$p$}-adic {T}eichm\"{u}ller space and {S}iegel halfspace},
  Study group on ultrametric analysis, 9th year: 1981/82, {N}o. 2, Inst. Henri
  Poincar\'{e}, Paris, 1983, pp.~Exp. No. 26, 15.

\bibitem[GH88]{GerritzenHerrlich_extendedSchottky}
L.~Gerritzen and F.~Herrlich, \emph{The extended {S}chottky space}, J. Reine
  Angew. Math. \textbf{389} (1988), 190--208.

\bibitem[GRW16]{GRW}
Walter Gubler, Joseph Rabinoff, and Annette Werner, \emph{Skeletons and
  tropicalizations}, Adv. Math. \textbf{294} (2016), 150--215.

\bibitem[Hej75]{Hejhal_Schottky&Teichmueller}
Dennis~A. Hejhal, \emph{On {S}chottky and {T}eichm\"{u}ller spaces}, Advances
  in Math. \textbf{15} (1975), 133--156.

\bibitem[Her87]{Herrlich_Teichmueller}
Frank Herrlich, \emph{Nichtarchimedische {T}eichm\"{u}llerr\"{a}ume}, Nederl.
  Akad. Wetensch. Indag. Math. \textbf{49} (1987), no.~2, 145--169.

\bibitem[Her90a]{Herrlich_extendedTeichmuellerC}
\bysame, \emph{The extended {T}eichm\"{u}ller space}, Math. Z. \textbf{203}
  (1990), no.~2, 279--291.

\bibitem[Her90b]{Herrlich_extendedTeichmueller}
\bysame, \emph{The non-{A}rchimedean extended {T}eichm\"{u}ller space},
  {$p$}-adic analysis ({T}rento, 1989), Lecture Notes in Math., vol. 1454,
  Springer, Berlin, 1990, pp.~256--266.

\bibitem[Her19a]{Herr_productformula}
Leo Herr, \emph{The {Log} {Product} {Formula}}, arXiv:1908.04936 [math] (2019).

\bibitem[Her19b]{Herrlich_personalcommunication}
Frank Herrlich, \emph{Personal communication}, 2019.

\bibitem[HKP18]{HolmesKassPagani}
David Holmes, Jesse~Leo Kass, and Nicola Pagani, \emph{Extending the double
  ramification cycle using {J}acobians}, Eur. J. Math. \textbf{4} (2018),
  no.~3, 1087--1099.

\bibitem[Hol19]{Holmes}
David Holmes, \emph{Extending the double ramification cycle by resolving the
  {Abel}-{Jacobi} map}, Journal of the Institute of Mathematics of Jussieu
  (2019), 1--29.

\bibitem[Ich00]{Ichikawa_Teichmuellermodularforms}
Takashi Ichikawa, \emph{Generalized {T}ate curve and integral {T}eichm\"{u}ller
  modular forms}, Amer. J. Math. \textbf{122} (2000), no.~6, 1139--1174.

\bibitem[Kat89]{Kato_logstr}
Kazuya Kato, \emph{Logarithmic structures of {F}ontaine-{I}llusie}, Algebraic
  analysis, geometry, and number theory ({B}altimore, {MD}, 1988), Johns
  Hopkins Univ. Press, Baltimore, MD, 1989, pp.~191--224.

\bibitem[Kat94]{Kato_toricsing}
\bysame, \emph{Toric singularities}, Amer. J. Math. \textbf{116} (1994), no.~5,
  1073--1099.

\bibitem[Kat00]{Kato_logsmoothcurves}
Fumiharu Kato, \emph{Log smooth deformation and moduli of log smooth curves},
  Internat. J. Math. \textbf{11} (2000), no.~2, 215--232.

\bibitem[KKN08]{KKN_logAVI}
Takeshi Kajiwara, Kazuya Kato, and Chikara Nakayama, \emph{Logarithmic abelian
  varieties. {I}. {C}omplex analytic theory}, J. Math. Sci. Univ. Tokyo
  \textbf{15} (2008), no.~1, 69--193.

\bibitem[Knu83]{Knudsen_projectivityII}
Finn~F. Knudsen, \emph{The projectivity of the moduli space of stable curves.
  {II}. {T}he stacks {$M_{g,n}$}}, Math. Scand. \textbf{52} (1983), no.~2,
  161--199.

\bibitem[Koe14]{Koebe_UniformisierungIV}
Paul Koebe, \emph{\"{U}ber die {U}niformisierung der algebraischen {K}urven.
  {IV}}, Math. Ann. \textbf{75} (1914), no.~1, 42--129.

\bibitem[KP19]{KassPagani}
Jesse~Leo Kass and Nicola Pagani, \emph{The stability space of compactified
  universal {J}acobians}, Trans. Amer. Math. Soc. \textbf{372} (2019), no.~7,
  4851--4887.

\bibitem[L{\"u}t16]{Luetkebohmert_book}
Werner L{\"u}tkebohmert, \emph{Rigid geometry of curves and their {J}acobians},
  Ergebnisse der Mathematik und ihrer Grenzgebiete. 3. Folge. A Series of
  Modern Surveys in Mathematics [Results in Mathematics and Related Areas. 3rd
  Series. A Series of Modern Surveys in Mathematics], vol.~61, Springer, Cham,
  2016.

\bibitem[MMU{\etalchar{+}}20]{MMUVW}
Margarida Melo, Sam Molcho, Martin Ulirsch, Filippo Viviani, and Jonathan Wise,
  \emph{Tropicalizing the universal jacobian: a logarithmic perspective}, 2020,
  Manuscript in preparation.

\bibitem[Moc99]{Mochizuki_padicTeichmueller}
Shinichi Mochizuki, \emph{Foundations of {$p$}-adic {T}eichm\"{u}ller theory},
  AMS/IP Studies in Advanced Mathematics, vol.~11, American Mathematical
  Society, Providence, RI; International Press, Cambridge, MA, 1999.

\bibitem[Moc02]{Mochizuki_intropadicTeichmueller}
\bysame, \emph{An introduction to $p$-adic {T}eichm\"uller theory},
  Cohomologies $p$-adiques et applications arithm\'etiques (I) (Berthelot
  Pierre, Fontaine Jean-Marc, Illusie Luc, Kato Kazuya, and Rapoport Michael,
  eds.), Ast\'erisque, no. 278, Soci\'et\'e math\'ematique de France, 2002,
  pp.~1--49 (en).

\bibitem[Mum72]{Mumford_uniformization}
David Mumford, \emph{An analytic construction of degenerating curves over
  complete local rings}, Compositio Math. \textbf{24} (1972), 129--174.

\bibitem[MUW17]{MUW}
Martin Moeller, Martin Ulirsch, and Annette Werner, \emph{Realizability of
  tropical canonical divisors}, arXiv:1710.06401 [math] (2017), to appear in J.
  Eur. Math. Soc. (JEMS).

\bibitem[MW17]{MarcusWise}
Steffen Marcus and Jonathan Wise, \emph{Logarithmic compactification of the
  {Abel}-{Jacobi} section}, arXiv:1708.04471 [math] (2017).

\bibitem[MW18]{MolchoWise_logPic}
Samouil Molcho and Jonathan Wise, \emph{The logarithmic {Picard} group and its
  tropicalization}, arXiv:1807.11364 [math] (2018).

\bibitem[Ogu18]{Ogus_logbook}
Arthur Ogus, \emph{Lectures on logarithmic algebraic geometry}, Cambridge
  Studies in Advanced Mathematics, vol. 178, Cambridge University Press,
  Cambridge, 2018.

\bibitem[Ols03]{Olsson_LOG}
Martin~C. Olsson, \emph{Logarithmic geometry and algebraic stacks}, Ann. Sci.
  \'{E}cole Norm. Sup. (4) \textbf{36} (2003), no.~5, 747--791.

\bibitem[PT20]{PoineauTurchetti_SchottkyoverZ}
J{\'e}r{\^o}me Poineau and Daniele Turchetti, \emph{Schottky spaces and
  universal {M}umford curves over {$\mathbb{Z}$}}, 2020, Manuscript in
  preparation.

\bibitem[Rab12]{Rabinoff_Newtonpolygons}
Joseph Rabinoff, \emph{Tropical analytic geometry, {N}ewton polygons, and
  tropical intersections}, Adv. Math. \textbf{229} (2012), no.~6, 3192--3255.

\bibitem[Ran19a]{Ranganathan_logGWexpansions}
Dhruv Ranganathan, \emph{Logarithmic {Gromov}-{Witten} theory with expansions},
  arXiv:1903.09006 [math] (2019).

\bibitem[Ran19b]{Ranganathan_productformula}
\bysame, \emph{A note on cycles of curves in a product of pairs},
  arXiv:1910.00239 [math] (2019).

\bibitem[Rei19]{Reinecke_infinitelevel}
Emanuel Reinecke, \emph{The cohomology of the moduli space of curves at
  infinite level}, arXiv:1911.07392 [math] (2019).

\bibitem[RSPW19a]{RSPWI}
Dhruv Ranganathan, Keli Santos-Parker, and Jonathan Wise, \emph{Moduli of
  stable maps in genus one and logarithmic geometry, {I}}, Geom. Topol.
  \textbf{23} (2019), no.~7, 3315--3366.

\bibitem[RSPW19b]{RSPWII}
\bysame, \emph{Moduli of stable maps in genus one and logarithmic geometry,
  {II}}, Algebra Number Theory \textbf{13} (2019), no.~8, 1765--1805.

\bibitem[Sch15]{Scholze_torsion}
Peter Scholze, \emph{On torsion in the cohomology of locally symmetric
  varieties}, Ann. of Math. (2) \textbf{182} (2015), no.~3, 945--1066.

\bibitem[Ser03]{Serre_trees}
Jean-Pierre Serre, \emph{Trees}, Springer Monographs in Mathematics,
  Springer-Verlag, Berlin, 2003, Translated from the French original by John
  Stillwell, Corrected 2nd printing of the 1980 English translation.

\bibitem[Thu07]{Thuillier_toroidal}
Amaury Thuillier, \emph{G\'{e}om\'{e}trie toro\"{\i}dale et g\'{e}om\'{e}trie
  analytique non archim\'{e}dienne. {A}pplication au type d'homotopie de
  certains sch\'{e}mas formels}, Manuscripta Math. \textbf{123} (2007), no.~4,
  381--451.

\bibitem[Uli17a]{Ulirsch_functroplogsch}
Martin Ulirsch, \emph{Functorial tropicalization of logarithmic schemes: the
  case of constant coefficients}, Proc. Lond. Math. Soc. (3) \textbf{114}
  (2017), no.~6, 1081--1113.

\bibitem[Uli17b]{Ulirsch_tropisquot}
\bysame, \emph{Tropicalization is a non-{A}rchimedean analytic stack quotient},
  Math. Res. Lett. \textbf{24} (2017), no.~4, 1205--1237.

\bibitem[Uli19]{Ulirsch_nonArchArtin}
\bysame, \emph{Non-{A}rchimedean geometry of {A}rtin fans}, Adv. Math.
  \textbf{345} (2019), 346--381.

\bibitem[Viv13]{Viviani_tropvscomp}
Filippo Viviani, \emph{Tropicalizing vs. compactifying the {T}orelli morphism},
  Tropical and non-{A}rchimedean geometry, Contemp. Math., vol. 605, Amer.
  Math. Soc., Providence, RI, 2013, pp.~181--210.

\bibitem[Yu16]{Yue_logCYI}
Tony~Yue Yu, \emph{Enumeration of holomorphic cylinders in log {C}alabi-{Y}au
  surfaces. {I}}, Math. Ann. \textbf{366} (2016), no.~3-4, 1649--1675.

\bibitem[Yu18]{Yu_Gromovcompactness}
\bysame, \emph{Gromov compactness in non-archimedean analytic geometry}, J.
  Reine Angew. Math. \textbf{741} (2018), 179--210.

\bibitem[Yu20]{Yue_logCYII}
\bysame, \emph{Enumeration of holomorphic cylinders in log {Calabi}-{Yau}
  surfaces. {II}. {Positivity}, integrality and the gluing formula},
  arXiv:1608.07651 [math] (2020), to appear in Geometry \& Topology.

\end{thebibliography}

\end{document}